\documentclass[11pt,reqno,a4paper]{amsart}

\usepackage{graphicx}
\usepackage{color}
\usepackage[USenglish]{babel}
\usepackage{latexsym}
\usepackage{amsmath}
\usepackage{amsfonts}
\usepackage{amssymb}
\usepackage{comment}
\usepackage{esint}
\usepackage[all]{xy}

\newcommand{\grad}{\nabla}
\newcommand{\lapl}{\Delta}

\renewcommand{\to}{\rightarrow}
\newcommand{\pa}{\partial}

\newcommand{\dsp}{\displaystyle}
\newcommand{\noi}{\noindent}
\renewcommand{\dfrac}{\displaystyle\frac}
\newcommand{\finedim}{\hspace{\fill}$\square$}
\newcommand{\intbar}{\mathop{\int\makebox(-15.5,0){\rule[6pt]{.7em}{0.3pt}}\kern-6pt}\nolimits}

\newcommand{\ii}{\infty}

\newcommand{\eps}{\varepsilon}
\newcommand{\dt}{\delta}

\newcommand{\al}{\alpha}

\newcommand{\sg}{\sigma}

\newcommand{\om}{\Omega}

\newcommand{\graf}[1]{\left\{\begin{array}{ll}#1\end{array}\right.}

\renewcommand{\b }{\beta }

\newcommand{\D }{\Delta }

\renewcommand{\L }{\Lambda }
\newcommand{\lm }{\lambda }

\renewcommand{\O }{\Omega }

\def\p{\partial}

\newcommand{\be}{\begin{equation}}
\newcommand{\ee}{\end{equation}}
\newcommand{\beq}{\begin{equation}}
\newcommand{\eeq}{\end{equation}}

\newcommand{\R}{\mathbb{R}}

\newcommand{\N}{\mathbb{N}}

\newcommand{\dis}{\displaystyle}

\newcommand{\vn}{v_n}

\newcommand{\bns}{\alpha_{n,\sigma}}

\newcommand{\bis}{\alpha_{\infty,\sigma}}

\newtheorem{theorem}{Theorem}[section]
\newtheorem{proposition}[theorem]{Proposition}
\newtheorem{definition}{Definition}[section]
\newtheorem{corollary}[theorem]{Corollary}
\newtheorem{remark}[theorem]{Remark}
\newtheorem{example}[theorem]{Example}

\newtheorem{lemma}[theorem]{Lemma}

\newcommand{\bpr}{\begin{proposition}}
\newcommand{\epr}{\end{proposition}}
\newcommand{\bex}{\begin{example}\rm}
\newcommand{\eex}{\end{example}}
\newcommand{\brm}{\begin{remark}\rm}
\newcommand{\erm}{\end{remark}}
\newcommand{\bdf}{\begin{definition}\rm}
\newcommand{\edf}{\end{definition}}
\newcommand{\bte}{\begin{theorem}}
\newcommand{\ete}{\end{theorem}}
\newcommand{\ble}{\begin{lemma}}
\newcommand{\ele}{\end{lemma}}
\newcommand{\bco}{\begin{corollary}}
\newcommand{\eco}{\end{corollary}}
\newcommand{\mycomment}[1]{}

\allowdisplaybreaks

\numberwithin{equation}{section}

\begin{document}
\title[Sharp Harnack type inequalities]
{Blow up and concentration without quantization:\\ sharp Harnack type inequalities.}

\author{Daniele Bartolucci, Paolo Cosentino*, Lina Wu}

\address{Daniele Bartolucci, Department of Mathematics, University of Rome {\it ``Tor Vergata"} \\  Via della ricerca scientifica n.1, 00133 Roma, Italy. }
\email{bartoluc@mat.uniroma2.it}

\address{Paolo Cosentino, Department of Mathematics, University of Rome {\it ``Tor Vergata"} \\  Via della ricerca scientifica n.1, 00133 Roma, Italy. }
\email{paolocosentino97@gmail.com}

\address{Lina Wu, School of Mathematics and Statistics, Beijing Jiaotong University, Beijing, 100044, China}
\email{lnwu@bjtu.edu.cn}

\begin{abstract}
Motivated by the Onsager statistical mechanics description of turbulent Euler flows with point singularities,
we refine the blow up analysis for sequences of solutions of a class of perturbed singular Liouville equations which share
the phenomenon of ``blow up and concentration without quantization". The problem is delicate because we are dealing with
the exact threshold value above which one meets
the well known ``concentration without quantization" phenomenon, as recently pushed forward in [C.S. Lin, G. Tarantello, {\em C. R. Math. Acad. Sci. Paris} (2016)] and in
[Y. Lee, C.S. Lin, G. Tarantello, W. Yang, {\em Comm. PDE.} (2017)].
We overcome these difficulties by developing different strategies for ``fast" and ``slow" blow up, combining refined Pohozaev identities, elliptic estimates and ``Sup+CInf" inequalities in the same spirit of [C.C. Chen, C.S. Lin, {\em Comm. An. Geom.} (1998)]. In the process, we also derive a new sharp Harnack type inequality by using a moving-plane method, following the ideas in [D. Bartolucci, C.C. Chen, C.S. Lin, G. Tarantello, {\em Comm. P.D.E.} (2002)].
\end{abstract}

\thanks{*Corresponding author: paolocosentino97@gmail.com}



\maketitle
\section{Introduction}

Motivated by the Onsager statistical mechanics description of turbulent Euler flows (\cite{clmp2,ESp})
interacting with a macroscopic fixed vortex, we have recently pursued (\cite{BCW1})
the blow up analysis of the perturbed singular Liouville equation,
\begin{equation}
 \label{eqbase}
 \graf{
 -\D v_n=H_n e^{\dis v_n} \quad \text{in}\quad B_1,\\
 \hfill \\
 \int_{B_1}\left({\epsilon_n^2+|x|^2}\right)^{\bns} e^{\dis v_n}\leq C, }
\end{equation}
\begin{equation}\label{lambdan}
 \lambda_n=\int_{B_1}\left({\epsilon_n^2+|x|^2}\right)^{\bns}K_n e^{\dis v_n}\to\lambda_\infty>0,
\end{equation}
\begin{equation*}
 \bns:= \tfrac{\lambda_n}{4\pi}\sigma, \quad \bis:= \tfrac{\lambda_\infty}{4\pi}\sigma,
\end{equation*}
\begin{equation}
 \label{H_n}
H_n(x)=\left({\epsilon_n^2+|x|^2}\right)^{\bns}K_n(x),\quad \sigma>0, \quad\epsilon_n\to 0,
\end{equation}
where we assume that $K_n$ satisfies
\begin{equation}
 \label{K+}
 K_n\geq 0, \,\, K_n\in C^{1}(\overline B_1),\quad  K_n\to K \quad \text{in}\,\,C^1(\overline{B}_1).
\end{equation}
To simplify notation, we shall sometimes write,
$$
h_{n}(x)=\left({\epsilon_n^2+|x|^2}\right)^{\frac12}\qquad\mbox{and}\qquad H_n=h_n^{2\bns}K_n.
$$

We will often assume that there exists a sequence of points $\{x_n\}\subset B_1$ such that
\begin{align}
\label{blowup:hyp}
 x_n\to 0 \,\,\,\,\text{and}&\,\,\,\, \underset{B_1}\sup\, v_n= v_n(x_n)\to +\infty, \,\,\text{as}\,\, n\to +\infty,
\end{align}
and
\begin{align}
\label{bound}
 \mbox{for any $r\in (0,1)$ there exists $C_r>0$ such that} \max\limits_{\overline{B_1\setminus B_r}}v_n\leq C_r.
\end{align}

\begin{remark}
From the point of view of conical singularities (\cite{bt}), the limit $\epsilon_n\to0^+$ yields a weight in \eqref{H_n} proportional to
$|x|^{2\al_{\sg,\ii}}$, $\al_{\sg,\ii}=\frac{\lm_\ii}{4\pi}\sg$, which describes a conical singularity at the origin of order $\al_{\sg,\ii}$.
\end{remark}

In an attempt to analyze the high energy limit of these configurations, we have shown in \cite{BCW1} that
solutions satisfying \eqref{blowup:hyp} are characterized, under suitable assumptions, by
a peculiar new phenomenon which we called ``blow-up and concentration without quantization",
which lies somewhere between the classical blow up-concentration-quantization phenomenon (\cite{bt}, \cite{bm}) and the blow up without concentration
as pushed forward in \cite{det,llty,LinTarantello,os}, see Theorem \ref{massboundarycontrol}.
In other words, we have blow up and concentration but, unlike the case of pure conical singularities (\cite{bt}),
the mass associated to the singular blow up point
is no longer constrained to attain a fixed value, being allowed instead to take values in some non degenerate interval of real numbers.
In fact, what we see via Theorem \ref{massboundarycontrol} is that the blow up of the stream function at the origin, i.e.
\eqref{blowup:hyp}, induces the concentration of the vorticity $\rho_n:=\tfrac{H_ne^{v_n}}{\int_{B_1} H_ne^{v_n}}$ exactly at the fixed location of the counter rotating vortex,
that is, a sort of collision of vortices of opposite sign. It is expected that these configurations should not arise as
thermodynamic equilibrium states, i.e. $\rho_n$ should not maximize the entropy at fixed energy $E$, in the sense of \cite{clmp2} (see also \cite{BCW1}),
at least for any $E$ large enough. Remark that, as far as $\lm_\ii\leq \frac{4\pi}{\sg}$, we have that  $\sg\leq \frac12$
is a necessary condition for \eqref{blowup:hyp}, see Theorem \ref{Concentration}. Thus, it seems interesting to
understand which are the values of $\lm_\ii$ for which collisions of this kind are really allowed and possibly compare the corresponding
entropy with that of the entropy maximizers as described in \cite{clmp2} as later refined in \cite{bl2}.
In \cite{BCW1} we have made a first step in this direction, providing a refined description of the bubbling
behavior of solutions of \eqref{eqbase}-\eqref{blowup:hyp} in a situation where we lack quantization, in particular in the regime
$$
\lm_n\to \lm_\ii\in \left[8\pi,\tfrac{4\pi}{\sg}\right),\quad \sg\in \left(0,\tfrac12\right)
$$
which, by Theorem \ref{massboundarycontrol}, implies $\lm_\ii<16\pi$ and, from the point of view of geometric singularities,
$$\frac{\lm_n}{4\pi}\sg=\al_{n,\sg}\to \al_{\ii,\sg}=\frac{\lm_\ii}{4\pi}\sg<1.$$ Remark that as far as \eqref{blowup:hyp} and
$\lm_\ii<\frac{4\pi}{\sg}$ holds true there is no loss of generality in assuming $\sg<\tfrac12$, since if $\sg=\frac12$
we would have $\lm_\ii<8\pi$ which implies a uniform bound from above for $v_n$, see Theorem \ref{Concentration}.
Actually, in particular due to $\al_{\ii,\sg}<1$, we could exploit in \cite{BCW1} a recently derived ``Sup+CInf" inequality (\cite{BCW0})
and, since $\lm_\ii<16\pi$, find that either 
$$\lm_\ii=8\pi,$$ with two different asymptotic regimes
(which we call ``fast" blow up, see Cases (I) and (II) in Theorem \ref{profile-new}),
or 
$$\lm_\ii\in (8\pi,\tfrac{8\pi}{1-2\sg}], \,\,\text{if}\,\, \sigma\in(0,\tfrac{1}{4}),\qquad \text{or}\qquad \lm_\ii\in(8\pi,\tfrac{4\pi}{\sigma}),\,\, \text{if}\,\, \sg\in[\tfrac{1}{4},\tfrac12),$$
(which we call ``slow" blow up, see Case (III) in Theorem \ref{profile-new}) corresponding to the asymptotically radial (\cite{GM1}) profile described in \cite{llty},
see Theorem \ref{profile-new} below for a detailed statement of this result.\\

It is our aim here to make further substantial progress on the program initiated in \cite{BCW1}, in particular  addressing a substantial part of the borderline regime
\begin{equation}\label{lm:intro.1}
    \lm_n\to \lm_\ii=\frac{4\pi}{\sg}\;\mbox{\rm  or equivalently }\al_{\ii,\sg}=1,
\end{equation}
$\mbox{with}\,\,\sg\in \left(0,\tfrac12\,\right], \,\,\mbox{under our blow up assumptions}$.\\
We face some problems due to the fact that $\al_{\ii,\sg}=1$. First of all, in order to provide a refined description of the possible asymptotic regimes, we actually can no longer rely on the ``Sup+CInf" inequality in \cite{BCW0}.
This is far from being a technical point, since any inequality of ``Sup+CInf" type implies that $C_1\inf\limits_{B_1}v_n$ is controlled by
$-v_n(x_n)+C_2$,  whence in particular that $v_n(x_n)\to +\ii$ implies concentration,
which is false in general for solutions of \eqref{eqbase} as far as $\bis>1$, see \cite{llty} and references quoted therein for more details
about this point. In particular, the standard argument used to prove the quantization in the ``fast" blow up (\cite{ls},\cite{BCW1}) cannot be applied. However, to overcome this problem, we use two different arguments. In one of the asymptotically ``fast" blow ups, we will pursue a refined analysis of the Pohozaev identity, see Theorem \ref{quant:imp}, showing that $\lm_\ii= 8\pi$ and
 $\sg=\frac12$. In other words, limited to this case, that term in the Pohozaev identity which in general contributes with the ``non quantized mass" (\cite{BCW1}), is actually quantized. However, it turns out that this same argument fails for the second ``fast" blow up, since the analysis of that term in the Pohozaev identity
 yields once again a value of the mass which could span a full interval of real numbers (see Theorem \ref{quant:imp}). We solve this
 problem by a careful argument based on elliptic estimates, yielding a bound from above for an otherwise indeterminate form (see Lemma \ref{lem:barwn}). As a consequence we can prove that in fact only one of the extremal values of that interval is allowed, see Proposition \ref{caseII.quant}.\\

Independently of the preceding quantization results, we also establish a new class of sharp Harnack-type inequalities by means of a refined moving-plane argument, in the spirit of \cite{bclt,bls,T1}, see Theorem  \ref{supinfI}. As a consequence, we obtain the following  ``Sup+CInf" inequality, which is of independent interest. As far as this result is concerned, in view of \eqref{K+} and \eqref{Kzero}, it is well known that there is no loss of generality in assuming that
\begin{equation*}
\min\limits_{\overline{B_{1}}} K\geq a>0
\end{equation*}
and
\begin{equation*}
    \sup_{\overline{B_{1}}}\|\nabla K_n\|_\ii\leq A,
\end{equation*}
for some uniform constant $A>0$. 

\begin{corollary}\label{cor:supinf}
Let $v_n$ be any sequence of solutions of \eqref{eqbase}, \eqref{lambdan}, \eqref{H_n}, \eqref{K+}.
Assume furthermore that,
$$
\max\limits_{\pa B_1} v_n-\min\limits_{\pa B_1} v_n\leq \L,
$$
and
$$
\bns=\frac{\lm_n}{4\pi} \sg\leq 1.
$$
There exists $C=C(a,A,\|K\|_\ii,\L)>0$, such that
\begin{equation}\label{sup-inf:0}
\tfrac{1-\bns}{1+\bns}\sup\limits_{B_\frac12} v_n+\inf\limits_{B_1}v_n\leq C,
\end{equation}
for $n$ large enough.
\end{corollary}
Since any inequality of ``Sup+CInf" type surely fails as far as $\bns\equiv \bis>1$ (\cite{llty}), then
\eqref{sup-inf:0} is interesting as it encodes this fact, reducing just to $\inf\limits_{B_1}v_n\leq C$ whenever $\bns\equiv 1$.
Actually, we cannot claim that \eqref{sup-inf:0} is stronger than the ``Sup+CInf" in \cite{BCW0}, which was derived with
no boundary conditions assumptions of any sort.\\

Concerning the ``slow" blow-up, the newly derived Harnack type inequalities
 degenerate at the critical threshold  $\al_{n,\sg}=1$, see \eqref{sup-inf:0} and \eqref{sup+inf-III} below.
 In particular, once again we miss the decay needed to prove that, even just in the relatively easier situation where $\lm_\ii<16\pi$,
 the mass outside the radial bubble vanishes in the limit. Therefore we need another argument here, which relies on the fact that this
 asymptotically radial blow up limit always comes with a mass larger than $8\pi$. Thus, limited to this case,
 we can exploit a careful adaptation of an argument in \cite{CLin0} (see also \cite{cosentino2025}), showing that a ``Sup+CInf" inequality holds for some $C$ large enough, see
 Lemma \ref{lemmasup+Cinfspecialcase}.
 It turns out that this allows one to conclude in a classical fashion (\cite{ls}) that the mass outside the ``radial bubble" vanishes in the limit, see
 Lemma \ref{LS-quant}. Defining
 $$
 \dt_n^{2(1+\bns)}:=e^{-\dsp v_n(x_n)}
 $$
we summarize these results with the following refinement of Theorem \ref{profile-new}.
\begin{theorem}\label{thm:13}
Let $v_n$ be any sequence of solutions of \eqref{eqbase}, \eqref{lambdan}, \eqref{H_n}, and \eqref{K+}.
Assume furthermore \eqref{blowup:hyp}, \eqref{bound},
$$
\max\limits_{\pa B_1} v_n-\min\limits_{\pa B_1} v_n\leq \L,
$$
and
$$
\bns=\frac{\lm_n}{4\pi} \sg\leq 1,\quad\quad \bis=\frac{\lm_\ii}{4\pi} \sg=1.
$$
Then, possibly along a subsequence, only one of the following alternatives may occur:
$$\text{either }\quad \sg=\frac14\,\,\,\text{and}\,\,\, \lm_\ii=16\pi,$$
$$
\text{or }\,\,\, \sg=\frac12\,\,\, \text{and}\,\,\,\lm_\ii=8\pi, \,\,\,\text{in which case}\quad \frac{\epsilon_n}{\dt_n}\to +\ii, \quad \frac{\epsilon_n}{|x_n|}\to +\ii,
$$
\begin{equation}\label{III-intro}
\text{or }\quad  \sg\in \left(\frac14,\frac12\right) \,\,\, \text{and}\,\,\,\ \lm_\ii\in(8\pi,16\pi), \,\,\,\text{in which case}\quad \frac{\epsilon_n}{\dt_n} \to \epsilon_0>0,\quad \frac{|x_n|}{\dt_n} \to 0.\quad
\end{equation}
Moreover, in case \eqref{III-intro} holds, $v_n$ takes the form \eqref{profilevtilde:H1-IIb}.
\end{theorem}

Remark that \eqref{III-intro} is a particular case of the ``slow" decay (which is Case (III) in Theorem \ref{profile-new}) where we also
come up with the profile estimates generalizing \eqref{profilevtilde:H1-IIb} to the case $\bis=1$. It is also interesting that the limiting mass
$8\pi$ only occurs in one of the ``fast" blow up limits, which is Case (I) in Theorem \ref{profile-new}. However, it turns out that
the efforts made so far are not enough to obtain a full description of the profiles in the cases $\lm_\ii=8\pi$ and $\lm_\ii=16\pi$.
We briefly illustrate this point. Consider, for example, the case $\sg=\frac12$, $\lm_\ii=8\pi$, then one knows that
$\int_{B_r}H_n e^{v_n} = 8\pi +o(1)$,
and $\int_{B_1\setminus B_r}H_n e^{v_n}=o(1)$ and the best possible estimate one could hope to find would be something like \eqref{profilev:H}, i.e.
a simple bubble which decays, after a suitable rescaling, as $-4\log(|y|)$ for $|y|\to +\ii$.  Perhaps surprisingly, even such an optimal pointwise decay would not suffice, by itself, to recover the vanishing of the mass outside the bubble. Indeed, after rescaling, the growth of the weight produces a logarithmically divergent upper bound. On the other hand, such strange behavior should be expected, since $\bis=1$ is the threshold value after which we face
blow up without concentration (\cite{llty}), a situation where having a minimal mass (Lemma \ref{minimalmasslemma}) of $8\pi$ at the blow up point does not
imply in general that $v_n \to -\ii$ locally uniformly far away from the blow up point. This rather elusive problem shows up again in another
subtle case for $\sg=\frac14$, i.e. $\lm_\ii=16\pi$, where it seems that one could face a nonstandard multiple bubbling phenomenon:
unlike the case of conical singularities of integer orders (\cite{bt2,dwz,KuLi,wz1,wz2}), the maximum points of two bubbles 
could, in principle, converge at the origin at different rates. We will discuss this issue in another paper.\\

This paper is organized as follows. In section \ref{sec2}, we state the main results from \cite{BCW1}. In section \ref{sec3} we use a Pohozaev argument to deduce a preliminary result of quantization in the ``fast" blow up, while in section \ref{sec4}, we refine these results, thereby completing the derivation of mass-quantization in the ``fast" regime. In section \ref{sectHarnackineq}, we derive a class of sharp Harnack inequalities and discuss their consequences. Lastly, in section \ref{sectslowblowup}, we treat the ``slow" blow up case, while the last sections contain the proofs of Theorem \ref{thm:13} and several technical lemmas.
\bigskip

\section{Blow up analysis: preliminary results from \cite{BCW1}.}\label{sec2}
Here and in the rest of this paper, we will often pass to subsequences which will not be relabeled.
Let us recall some results from \cite{BCW1}.

\begin{lemma}[The Minimal Mass Lemma (\cite{BCW1})]\label{minimalmasslemma}$\,$\\
Let $v_n$ be a sequence of solutions of \eqref{eqbase} and \eqref{lambdan} satisfying \eqref{blowup:hyp}, 
where $H_n$,  $K_n$ and $K$ satisfy \eqref{H_n} and \eqref{K+} respectively. Assume furthermore that,
\begin{equation*}
 H_ne^{\dis v_n} \rightharpoonup \eta \qquad \text{weakly in the sense of measures in}\, B_1,
\end{equation*}
for some bounded Radon measure $\eta$, where we define,
$$
\beta:=\eta(\{0\}).
$$
Then
\begin{equation}\label{Kzero}
 K(0)>0,
\end{equation}
and
\begin{align}
 \label{minimalmass}
\lm_\ii\geq \beta\geq8\pi.
\end{align}
\end{lemma}

\bigskip

Next, let us recall a ``concentration"-type result in \cite{BCW1} under the following assumptions,
\[
 \text{for any}\,\,r\in(0,1),\exists \, C_r>0, \,\,\text{such that}:
\]
\begin{equation}
    \label{bounded}
   \underset{\overline{B_1\backslash B_r}}{\max}\, v_n\leq C_r,
\end{equation}
\begin{equation}
    \label{explosion}
    \underset{\overline{B_r}}\max\, v_n\to\infty.
\end{equation}

\begin{theorem}[Concentration (\cite{BCW1})]\label{Concentration}\hfill\\
Let $v_n$ be a sequence of solutions of \eqref{eqbase} and \eqref{lambdan}, where $K_n$ satisfies \eqref{K+}.
Assume furthermore that  \eqref{bounded}, \eqref{explosion} holds true. Then \eqref{minimalmass} holds true as well and if
\begin{equation}
 \label{conditionmixed}
 \sigma>0\,\,\,\,\,\,\text{and}\,\,\,\,\,\,\, \lambda_\infty\leq\tfrac{4\pi}{\sigma},
\end{equation}
then necessarily $\sg\leq \frac12$ and there exists a subsequence of $v_n$ such that
\begin{align}
 \label{-infinity}
 &v_{n}\to-\infty,\qquad\qquad \text{as}\,\,n\to\infty,\,\text{uniformly on compact sets of}\,\,B_1\backslash\{0\},
\\
 \nonumber
 &H_{n}e^{\dis v_{n}}\rightharpoonup \beta\delta_{p=0}, \quad\text{weakly in the sense of measures in}\,\, B_1,\\
 \label{mmbeta}
 & \text{where}\quad \lm_\ii\geq \beta\geq
            8\pi.
\end{align}
\end{theorem}

\medskip

As a consequence of Theorem \ref{Concentration}, in the same spirit of \cite{bt} as later refined in \cite{os}, we have the following concentration-compactness alternative:
\begin{theorem}[Concentration-Compactness]\label{Concentration-Compactness alternative}\hfill \\
Let $v_n$ be a sequence of solutions of \eqref{eqbase}, \eqref{lambdan}, where $K_n$ satisfies \eqref{K+}.
Assume furthermore that \eqref{conditionmixed} holds true.
Then, possibly along subsequence, one of the following alternatives is satisfied:
\begin{itemize}
 \item[(i)] $\underset{\om}\sup\, \big|v_{n}\big|\leq C_\O$, for every $\O\Subset B_1$;
 \item[(ii)] $\underset{\om}\sup\, v_{n}\to-\infty$, for every $\O\Subset B_1$;
 \item[(iii)] there exists a finite and nonempty set $S=\{q_1,\ldots,q_m\}\subset B_1$, $m\in\N$, and sequences of points
 $\{x_n^1\}_{n\in\N},\ldots,\{x_n^m\}_{n\in\N}\subset B_1$, such that $x_n^i\to q_i$ and $v_{n}(x_n^i)\to\infty$,
 for every $i\in\{1,\ldots,m\}$. Moreover,
 $\sup_\O\, v_{n}\to-\infty$
 on any compact set $\O\subset B_1\backslash S$ and $H_{n}e^{\dis v_{n}}\rightharpoonup \sum_{i=1}^m\beta_i\delta_{q_i}$
weakly in the sense of measures in $B_1$, with $\beta_i\in 8\pi\N$ if $q_i\neq0$, while if $q_i=0$ for some $i\in\{1,\ldots,m\}$ then $\beta_i$ satisfies \eqref{mmbeta}. In particular $\lm_\ii\geq 8\pi m$.
\end{itemize}
\end{theorem}

\begin{remark}\label{rem:concentrationgeneral}
We point out that the concentration and concentration-compactness analysis underlying Lemma \ref{minimalmasslemma} and Theorems \ref{Concentration} and \ref{Concentration-Compactness alternative} extends to the more general equation
\[\displaystyle
\graf{-\D v_n=(\epsilon_n^2+|x|^2)^{\alpha_n}K_n(x)e^{\dis v_n}\quad\text{in}\,\,\O,\\
\displaystyle
\int_{\O}(\epsilon_n^2+|x|^2)^{\alpha_n}e^{\dis v_n}\,dx\leq C,\\
\displaystyle
\lm_n=\int_{\O}(\epsilon_n^2+|x|^2)^{\alpha_n}K_n(x)e^{\dis v_n}\,dx\to\lm_\ii>0,}
\]
 where $\O$ is an open and bounded domain in $\R^2$, $0\in\O$, $\alpha_n$ is any sequence satisfying $\alpha_n\to\alpha_\infty\in(0,1]$, independently of $\lm_n$, and the remaining assumptions are unchanged. In this setting, the restriction $\alpha_\ii\leq1$ replaces \eqref{conditionmixed}. The proofs can be adapted as in \cite{BCW1}. We also refer to \cite{os} for a more general concentration-compactness framework. 
\end{remark}

\bigskip

The estimate about the value of the mass relative to the blow up at singular sources can be refined whenever
we have a mild control about the oscillation of $v_n$ at the boundary, as first noticed in \cite{yy}.

\begin{theorem}[Concentration without quantization (\cite{BCW1})] \label{massboundarycontrol} \hfill \\
Under the assumptions of Theorem \ref{Concentration} about $\vn$, $K_n$ and $\lm_n$, suppose in addition that $v_n$ satisfies,
 \begin{equation}
  \label{oscillation}
  \underset{\p B_1}\max\,v_n-\underset{\p B_1}\min\,v_n\leq \L.
 \end{equation}
 Then
\begin{equation}\label{beqlm}
 \beta\equiv \lm_\ii
\end{equation}
and, as far as \eqref{conditionmixed} is satisfied, we have that,
\begin{equation*}
0<\sg\leq \frac12\,,
\end{equation*}
\begin{equation}\label{boundsbeta}
8\pi\leq  \beta \leq 8\pi(1+\bis)=8\pi+2\sg\lm_\ii
\end{equation}
and in particular
\begin{equation}\label{quantized:poslm}
8\pi\leq \lm_\ii\leq\min\left\{\frac{8\pi}{1-2\sg},\frac{4\pi}{\sg}\right\}.
\end{equation}
\end{theorem}

\bigskip

We will need in the sequel three auxiliary functions, to be used
over and over during the blow up argument. To simplify the exposition, we list them hereafter. Moreover, here and in the rest of the paper we set,
$$
c_0=K(0).
$$
Under the assumptions of Lemma \ref{minimalmasslemma},
let us set,
\begin{equation*}
     \delta_n^{2(1+\bns)}=e^{-\dis {v_n(x_n)}}\to 0,\,\, \text{as} \,\, n\to\infty,
\end{equation*}
and
\[
 t_n=\max\{|x_n|,\delta_n\}\to 0,\,\, \text{as} \,\, n\to\infty.
\]

We are naturally led to analyze three different cases:
\begin{itemize}
 \item Case \;(I): there exists a subsequence such that
 \begin{equation}
 \label{Hypothesis:CaseI}
 \tfrac{\epsilon_n}{t_n}\to+\infty\,\,\text{as}\;\; n\to+\infty;
\end{equation}
 \item Case (II): there exists a constant $C_1>0$ such that
\begin{equation}
 \label{Hypothesis:CaseII}
 \frac{\epsilon_n}{t_n}\leq C_1,\,\,\text{for all}\,\,n\in\N
\end{equation}
 and, possibly along a subsequence, $\tfrac{|x_n|}{\delta_n}\to+\infty$, as $n\to+\infty$;
 \item Case (III): there exist constants $C_1>0$ and $C_2>0$ such that
 \begin{equation}
 \label{Hypothesis:CaseIII}
 \frac{\epsilon_n}{t_n}\leq C_1,\quad \frac{|x_n|}{\delta_n}\leq C_2,\,\,\text{for all}\,\,n\in\N.
\end{equation}
\end{itemize}

Concerning Case (I) we see that,
\begin{equation}
 \label{epsilondelta&epsilonxn}
 \frac{\epsilon_n}{\delta_n}\to+\infty\qquad\text{and}\qquad \frac{\epsilon_n}{|x_n|}\to+\infty,
\end{equation}
and, for some $r>0$ such that $\min\limits_{\overline B_r}K>0,$ we define,
\begin{equation}\label{bn1}
B^{(n)}=B_{\frac{r}{\epsilon_n}}(0),
\end{equation}
and
\begin{equation}\label{wn:IA}
 w_n(z)=v_n(\epsilon_n z)+2(1+\bns)\log\epsilon_n\quad\text{in}\quad B^{(n)},
\end{equation}
which therefore satisfies,
\begin{align}\label{eqwn:IA2}
\begin{cases}
-\D w_n=V_n(z)e^{\dis w_n} \quad\text{in}\quad B^{(n)},\\
V_n(z)=\left({1+|z|^2}\right)^{\bns} K_n(\epsilon_n z)\to
c_0\left({1+|z|^2}\right)^{\bis}\;\mbox{in}\;C^t_{loc}(\R^2),\\
\int\limits_{B^{(n)}}\left({1+|z|^2}\right)^{\bns}e^{\dis w_n}\leq C,\\
\int\limits_{B^{(n)}}V_n e^{\dis w_n}= \lm_n +o(1) = \lm_\ii +o(1),\\
w_n(\tfrac{x_n}{\epsilon_n})=2(1+\bns)\log(\tfrac{\epsilon_n}{\delta_n})\to +\infty.
\end{cases}
\end{align}

Concerning  Case (II) we notice  that
$\delta_n\leq|x_n|$, for $n$ large enough, whence
\begin{equation}
 \label{epsilonxn}
 \frac{\epsilon_n}{|x_n|}\leq C,\,\,\text{for all}\,\,n\in\N,
\end{equation}
for some $C>0$ constant. Thus, possibly along a subsequence,
in particular we have that,
\[
 \frac{\epsilon_n}{|x_n|}\to\bar \epsilon_0\geq  0,\quad \frac{x_n}{|x_n|}\to \bar z_0,\;\;|\bar z_0|=1,\quad \text{as}\,\,n\to+\infty.
\]
 In this situation we define,
\begin{equation}\label{dn1}
D^{(n)}=B_{\frac{r}{|x_{n}|}}(0),
\end{equation}
\begin{equation}\label{barwn:IB}
 \bar w_{n}(z)=v_n(|x_{n}| z)+2(1+\bns)\log(|x_{n}|)\quad\text{in}\quad D^{(n)},
\end{equation}
which satisfies
\begin{align}\label{eqwn:IB}
\begin{cases}
-\D \bar w_{n}=\bar V_{n}(z)e^{\dis \bar w_{n}} \quad\text{in}\quad D^{(n)},\\
\bar V_{n}(z)=\left({\frac{\epsilon_n^2}{|x_{n}|^2}+|z|^2}\right)^{\bns}K_n(|x_{n}| z)\to c_0(\bar \epsilon_0^2+|z|^2)^{\bis}\;\mbox{in}\;C^t_{loc}(\R^2),\\
\int\limits_{D^{(n)}}\left({\frac{\epsilon_n^2}{|x_{n}|^2}+|z|^2}\right)^{\bns} e^{\dis \bar w_n}\leq C,\\
\int\limits_{D^{(n)}}\bar V_n e^{\dis \bar w_n}= \lm_n +o(1) = \lm_\ii +o(1),\\
\bar w_{n}(\tfrac{x_{n}}{|x_{n}|})=2(1+\bns)\log(\tfrac{|x_{n}|}{\delta_{n}})\to +\infty.
\end{cases}
\end{align}

At last, concerning Case (III), we see that \eqref{Hypothesis:CaseIII}, 
together with the fact that $\tfrac{|x_n|}{\delta_n}$ is uniformly bounded, implies that
\begin{equation*}
 \frac{\epsilon_n}{\delta_n}\leq C,\,\,\text{for all}\,\,n\in\N,
\end{equation*}
for some constant $C>0$. Indeed we have that,
\begin{align*}
 \frac{\epsilon_n}{\delta_n}=\frac{\epsilon_n}{t_n}\frac{t_n}{\delta_n}\leq C_1\max\{1, C_2\}.
\end{align*}
Let us define
\begin{equation}
\label{dn2}
D_n=B_{\frac{r}{\delta_n}}(0),
\end{equation}
and possibly along a subsequence, assume without loss of generality that,
\begin{equation*}
 \frac{x_n}{\delta_n}\to y_0, \,\,\text{as}\,\,n\to+\infty
\end{equation*}
 and
\begin{equation*}
 \frac{\epsilon_n}{\delta_n}\to \epsilon_0, \,\,\text{as}\,\,n\to+\infty,
\end{equation*}
for some $y_0\in\R^2$ and $\epsilon_0\geq 0$. At this point, we define,
\begin{equation}
 \label{phicaseIA}
 \widetilde w_n(y)=v_n(\delta_n y)+2(1+\bns)\log \delta_n=v_n(\delta_n y)-v_n(x_n)\quad\text{in}\,\, D_n,
\end{equation}
which satisfies,
\begin{align}\label{equationgoodcase}
\begin{cases}
 -\D \widetilde w_n=\left({\frac{\epsilon_n^2}{\delta_n^2}+\left|y\right|^2}\right)^{\al_{n,\sg}}
 K_n(\delta_n y)e^{\dis \widetilde w_n} \qquad\text{in}\,\,D_n, \\
\widetilde{V}_n(y)= K_n(\delta_n y)\left({\frac{\epsilon_n^2}{\delta_n^2}+\left|y\right|^2}\right)^{\al_{n,\sg}}\to c_0( \epsilon_0^2+|y|^2)^{\al_{\ii,\sg}}\;\mbox{in}\;C^t_{loc}(\R^2),
 \\
\int_{D_n}\left({\frac{\epsilon_n^2}{\delta_n^2}+|y|^2}\right)^{\bns}e^{\dis \widetilde w_n}\leq C,\\
\int\limits_{D_n}\widetilde V_n e^{\dis \widetilde w_n}= \lm_n +o(1) = \lm_\ii +o(1),\\
\widetilde w_n(y)\leq\widetilde w_n(\tfrac{x_n}{\delta_n})=\max\limits_{D_n}\widetilde w_n=0.
\end{cases}
\end{align}

Corresponding to each one of the three cases above, in \cite{BCW1} we obtained sharp profile
estimates generalizing the known results in the regular case (i.e. $\al_{n,\sg}\equiv 0$, \cite{yy}) and singular case 
(i.e. $\epsilon_n\equiv 0$, \cite{bclt}).
We conclude this section with the statement of the main result in \cite{BCW1}. 
One of the allowed limiting profiles (which occurs in Case (III) below), corresponds after blow up to solutions of the following planar problem,
\begin{align}
\begin{cases}\label{profile:tildew}
 -\D \widetilde w=\widetilde V(y) e^{\dis \widetilde w} \quad\text{in}\quad \R^2, \\
\widetilde V(y)=c_0\left(\epsilon_0^2+|y|^2\right)^{\bis},\quad c_0>0,\\
\widetilde \beta := \int_{\R^2}\widetilde V e^{\dis \widetilde w}\leq\lm_\ii,\\
\widetilde w(y)\leq \widetilde w(y_0)=0,\quad y_0\in\R^2.
\end{cases}
\end{align}
See section \ref{sec:planar} below for more about \eqref{profile:tildew}.

\begin{theorem}[Profile estimates with lack of quantization (\cite{BCW1})]\label{profile-new} Under the hypothesis of Theorem \ref{massboundarycontrol},
assume furthermore that
$$
\sg\in\left(0,\frac12\right)\quad \mbox{and} \quad 8\pi\leq \lm_\ii< \frac{4\pi}{\sg},
$$
and set
\[
 t_n=\max\{\delta_n,|x_n|\}.
\]
Then, either,
\begin{itemize}
 \item \mbox{\rm (I):}  there exists a subsequence such that $\tfrac{\epsilon_n}{t_n}\to+\infty$,
\end{itemize}
in which case we have $\lm_\ii=8\pi$ and
\begin{equation}\label{profilev:H}
v_n(x)=\log\left(\dfrac{e^{\dis v_n(x_n)}}
{\left(1+\gamma_n\theta_n^{2(1+\bns)}\epsilon_n^{-\frac{\lm_n}{4\pi}}|x-{x_n}|^{\frac{\lm_n}{4\pi}}\right)^2}\right)+O(1),\quad x\in B_{r}(0);
\end{equation}
where
$\theta_n^{2(1+\al_{n,\sg})}=
\left(\tfrac{\epsilon_n}{\delta_n}\right)^{2(1+\al_{n,\sg})}\to +\ii$, $8 \gamma_n={(1+|\frac{x_n}{\epsilon_n}|^2)^{\bns} K_n({x_n})}$, or
\begin{itemize}
 \item \mbox{\rm (II):} there exists a subsequence such that $\tfrac{\epsilon_n}{t_n}\leq C$ and $\tfrac{|x_n|}{\delta_n}\to+\infty$,
\end{itemize}
in which case we have $\lm_\ii=8\pi$ and
\begin{equation}\label{profilev:H1}
v_n(x)=\log\left(\dfrac{e^{\dis v_n(x_n)}}
{\left(1+\bar\gamma_n\bar \theta_n^{2(1+\al_{n,\sg})}|x_n|^{-\frac{\lm_n}{4\pi}} |x-{x_n}|^{\frac{\lm_n}{4\pi}}\right)^2}\right)+O(1),\quad x\in B_{r}(0),
\end{equation}
where $\bar \theta_n^{2(1+\al_{n,\sg})}=\left(\tfrac{|x_{n}|}{\delta_n}\right)^{2(1+\al_{n,\sg})}\to +\ii$,
$8\bar\gamma_n={((\tfrac{\epsilon_n}{|x_n|})^2+1)^{\bns}}K_n(x_n)$, or
\begin{itemize}
 \item \mbox{\rm (III):} there exists a subsequence such that $\tfrac{\epsilon_n}{t_n}\leq C$ and
 $\tfrac{|x_n|}{\delta_n}\leq C$,
\end{itemize}
in which case, along a further subsequence if necessary, we have $\tfrac{\epsilon_n}{\delta_n}\to \epsilon_0\geq 0$ and\\ $\lm_\ii\in (8\pi,\tfrac{8\pi}{1-2\sg}]$, if $\sigma\in(0,\tfrac{1}{4})$, or $\lm_\ii\in(8\pi,\tfrac{4\pi}{\sigma})$, if $\sg\in[\tfrac{1}{4},\tfrac12)$ and
\begin{equation}\label{profilevtilde:H1-IIb}
 v_n(x)=v_n(x_n)+ \widetilde U_n(\delta_n^{-1}x)+O(1), \quad x\in B_{r}(0),
\end{equation}

where
\begin{equation*}
    \widetilde U_n(y)=\graf{\widetilde w(y)+O(1),\quad |y|\leq R, \\ -\frac{\lm_n}{2\pi}\log(|y|) + O(1),\quad R\leq  |y|\leq r\delta_n^{-1}}
\end{equation*}
and $\widetilde w$ is the unique solution of \eqref{profile:tildew}.
\end{theorem}

\bigskip
\bigskip

\section{A refinement of Theorem \ref{massboundarycontrol}. }\label{sec3}

We recall that, in view of \eqref{K+} and \eqref{Kzero}, after suitable rescaling if needed, there is no loss of generality in assuming that
\begin{equation}\label{defa}
\min\limits_{\overline{B_{1}}} K\geq a>0
\end{equation}
and
\begin{equation}
    \label{gradientbound}
    \sup_{\overline{B_{1}}}\|\nabla K_n\|_\ii\leq A,
\end{equation}
for some uniform constant $A>0$.
We will use $C$ to denote various constants which may possibly change even line to line.\\

Now, we prove the following refinement of Theorem \ref{massboundarycontrol}.

\begin{theorem}\label{quant:imp}
Under the hypothesis of Theorem \ref{massboundarycontrol},
assume furthermore that 
$$
\lm_n\leq\frac{4\pi}{\sg}\quad\text{and}\quad\lm_\ii=\frac{4\pi}{\sg}.
$$
If the condition of CASE (I) in Theorem \ref{profile-new} is satisfied, that is $\tfrac{\epsilon_n}{t_n}\to+\infty$, then
$$\sg=\frac12\,\,\, \text{and}\,\,\,\lm_\ii=8\pi.$$
If the condition of CASE (II) in Theorem \ref{profile-new} is satisfied, that is,
there exists a subsequence such that $\tfrac{\epsilon_n}{t_n}\to\bar\epsilon_0\geq 0$ and $\tfrac{|x_n|}{\delta_n}\to+\infty$, then we have\\
 $$\sg=\frac{1}{2}\left(1-{\frac{1}{1+\sqrt{1+\bar\epsilon_0^2}}}\right)\,\,\,
 \text{and} \,\,\,\lm_\ii=8\pi\left(1+\frac{1}{\sqrt{1+\bar\epsilon_0^2}}\right),$$
 for some $\bar\epsilon_0\geq0$.
\end{theorem}
\proof
We notice that by Theorem \ref{massboundarycontrol} we necessarily have,
\[
8\pi\leq\lm_\ii\leq\min\{\tfrac{8\pi}{1-2\sg},\tfrac{4\pi}{\sg}\}.
\]
However, if $\sigma\in(0,\tfrac{1}{4})$, then $\lm_\ii\leq\tfrac{8\pi}{1-2\sg}<\tfrac{4\pi}{\sg}$, which contradicts the hypothesis. Therefore, we are going to consider $\sg\in[\tfrac{1}{4},\tfrac{1}{2}]$.\\
We refine an argument first introduced in \cite{bt} as recently applied in this context in \cite{BCW1}.\\
Let us define $s_n$ to be the unique solution of,
\begin{equation*}
 \begin{cases}
  -\D s_n=0\quad &\text{in}\,\,B_1,\\
  s_n=v_n-\underset{\p B_1}\min\,v_n\quad &\text{on}\,\,\p B_1.
 \end{cases}
\end{equation*}
By standard elliptic estimates, we have that
\begin{equation*}
 \|s_n\|_\infty\leq C,
\end{equation*}
for a suitable positive constant $C$, and, along a subsequence, we may assume that $s_n\to s$ in $C^1_{loc}(B_1)$. Now, let us consider the function
\begin{equation*}
 \zeta_{n}(x)=v_n(x)-\underset{\p B_1}\min\,v_n-s_n(x),
\end{equation*}
which satisfies,
\begin{equation*}
 \begin{cases}
  -\D \zeta_n=\widehat W_n e^{\zeta_n}\quad \text{in}\quad B_1, \\
  \int_{B_1}\widehat W_ne^{\zeta_n}\leq C,\\
  \zeta_n=0 \quad \text{on}\quad \p B_1,
 \end{cases}
\end{equation*}
where
\begin{equation*}
 \widehat W_n(x)=H_n(x)e^{\omega_n(x)}
\end{equation*}
and
\begin{equation*}
 \omega_n(x)=s_n(x)+\underset{\p B_1}\min \,v_n.
\end{equation*}
We notice that,
\begin{equation*}
 \nabla\omega_n\to\nabla s\,\,\text{and}\,\, \nabla K_n\to\nabla K, \,\,\text{uniformly on compact sets of $B_1$}.
\end{equation*}
Moreover, since $\widehat W_n(x)e^{\dis \zeta_n}=H_ne^{\dis v_n}$, by Theorem \ref{Concentration}, we have that
\begin{equation}
 \label{vanishing}
 \widehat W_ne^{\dis \zeta_n}\to0\qquad \text{uniformly on compact sets of $B_1\backslash\{0\}$}
\end{equation}
and
\begin{equation}
 \label{concentrationwn}
 \widehat W_ne^{\dis \zeta_n}\rightharpoonup \beta\delta_{p=0}\qquad \text{weakly in the sense of measure in $B_1$},
\end{equation}
where we are using that $\beta$ satisfies \eqref{beqlm} and \eqref{quantized:poslm}. Moreover,
\begin{equation}
 \label{concentration2}
 h_{n}^{2\bns}e^{\dis \omega_n}e^{\dis \zeta_n}\rightharpoonup\tfrac{\lm_\ii}{K(0)}\delta_{p=0}\qquad \text{weakly in the sense of measure in $B_1$}.
\end{equation}
Let us set $f_n(x)=\widehat W_n(x)e^{\dis \zeta_n(x)}$, then by Green's representation formula,
\begin{equation*}
 \zeta_n(x)=-\tfrac{1}{2\pi}\int_{B_1}\ln (|x-y| )f_n(y)\,dy+\int_{B_1}R(x,y)f_n(y)\,dy,
\end{equation*}
where $R(x,y)$ is the regular part of Green's function associated to the Laplacian operator with Dirichlet boundary conditions on $B_1$.
Then, passing to the limit in the previous identity, we deduce that
\begin{equation}
 \label{wnconverges}
 \zeta_n(x)\to -\tfrac{\lm_\ii}{2\pi}\ln|x|+ g(x)\quad \text{in} \,\,C^1_{\rm loc}(B_1\backslash\{0\}),
\end{equation}
for some $g\in C^1(B_1)$. Let us set,
\[
 \zeta_0(x)= -\tfrac{\lm_\ii}{2\pi}\ln|x|+ g(x),
\]
we use the well known Pohozaev identity, for $r\in(0,1)$,
\begin{align}
 \nonumber
 \int_{\p B_r(0)}\Big((x,\nu)\tfrac{\dis |\grad \zeta_n|^2}{2}-&(\nu,\grad \zeta_n)(x,\grad \zeta_n)\Big)d\sigma= \\ \label{Pohozaevnwn}
 &=\int_{\p B_r(0)}(x,\nu)\widehat W_n e^{\dis \zeta_n}d\sigma-\int_{B_r}(2\widehat W_n+x\cdot\grad \widehat W_n)e^{\dis \zeta_n}\,dx.
\end{align}

First of all, observe that
\begin{align*}
 &-\int_{B_r}x\cdot\grad \widehat W_ne^{\dis \zeta_n}\,dx= \\
 &=-\int_{B_r}x\cdot\grad (K_ne^{ \omega_n})h_{n}^{2\bns}e^{\dis \zeta_n}\,dx-
 2\bns\int_{B_r}|x|^2\big(\epsilon_n^2+|x|^2\big)^{\bns-1}K_ne^{ \omega_n}e^{\dis \zeta_n}\,dx\\
 &=-\int_{B_r}\tfrac{x\cdot\grad (K_ne^{ \omega_n})}{K_ne^{ \omega_n}}\widehat W_ne^{\dis \zeta_n}\,dx-2\bns\int_{B_r}\widehat W_ne^{\dis \zeta_n}\,dx+
 2\bns\int_{B_r}\epsilon_n^2\big(\epsilon_n^2+|x|^2\big)^{\bns-1}K_ne^{\dis \omega_n}e^{\dis \zeta_n}\,dx \\
 &=-\int_{B_r}\tfrac{x\cdot\grad (K_ne^{ \omega_n})}{K_ne^{ \omega_n}}\widehat W_ne^{\dis \zeta_n}\,dx-2\bns\int_{B_r}\widehat W_ne^{\dis \zeta_n}\,dx+
 2\bns\int_{B_r}\tfrac{\epsilon_n^2}{\epsilon_n^2+|x|^2}\widehat W_ne^{\dis \zeta_n}\,dx.
\end{align*}

As in \cite{bt}, by using \eqref{vanishing}, \eqref{concentration2} and \eqref{wnconverges},  taking the limits as $n\to +\ii$ and $r\to 0$, we have that
$$
\int_{\p B_r(0)}\Big((x,\nu)\tfrac{\dis |\grad \zeta_n|^2}{2}-(\nu,\grad \zeta_n)(x,\grad \zeta_n)\Big)d\sigma=-\tfrac{\lm_\ii^2}{4\pi}+o(1),
$$
and
$$
\int_{B_r}\tfrac{x\cdot\grad (K_ne^{ \omega_n})}{K_ne^{ \omega_n}}\widehat W_ne^{\dis \zeta_n}\,dx=o(1), \qquad \int_{\p B_r(0)}(x,\nu)\widehat W_n e^{\dis \zeta_n}d\sigma=o(1),
$$
$$
2\int_{B_r}\widehat W_n e^{\dis \zeta_n}\,dx=2 \lm_\ii +o(1), \qquad 2\bns\int_{B_r}\widehat W_n e^{\dis \zeta_n}\,dx=2\bis \lm_\ii +o(1).
$$
In particular, taking firstly the limit as $n\to +\ii$ and then as $r\to 0$, we obtain that
\begin{equation}
    \label{Pohozaevextraterm}
    2\bns\int_{B_r}\tfrac{\epsilon_n^2}{\epsilon_n^2+|x|^2}\widehat W_ne^{\dis \zeta_n}\,dx=2\bns\int_{B_r}\tfrac{\epsilon_n^2}{\epsilon_n^2+|x|^2} H_ne^{\dis v_n}\,dx=-\tfrac{\lm_\ii^2}{4\pi}+2(1+\bis)\lm_\ii+o(1).
\end{equation}

At this stage, we want to estimate $2\bns\int_{B_r}\tfrac{\epsilon_n^2}{\epsilon_n^2+|x|^2}H_ne^{\dis v_n}\,dx$ and we discuss the different cases one by one.\\

CASE (I): $\tfrac{\epsilon_n}{\dt_n}\to+\infty$ and $\tfrac{\epsilon_n}{|x_n|}\to+\infty$.
\begin{remark}\label{rem:reg} Clearly the concentration compactness theory (\cite{bm}, \cite{yy})
for regular Liouville type equations can be applied to $w_n$ on any compact subset of $\R^2$, implying that
the blow up set relative to $w_n$ is finite and the limiting total mass cannot exceed $\lm_\ii$.
Actually, for any compact set $U\subset \R^2$
far away from the blow up set of $w_n$, since $v_n$ satisfies \eqref{oscillation}, by Lemma 2.2 in \cite{bclt} we have that,
\begin{equation}\label{osc:wnU}
  \underset{U}\max\,w_n-\underset{U}\min\,w_n\leq C_U.
\end{equation}
Therefore, in particular
\begin{equation*}
  \underset{\p B_{R}}\max\,w_n-\underset{\p B_{R}}\min\,w_n\leq C_{R},
\end{equation*}
for any $R$ small enough and, by the result in \cite{yy}, $0$ is a simple blow up point for $w_n$,
$$
V_n e^{\dis w_n}\rightharpoonup 8\pi \dt_{p=0}\qquad\text{weakly in the sense of measures of}\,\, B_R,
$$
and
$$
\int\limits_{B_{R}(0)}V_n e^{\dis w_n}=8\pi +o(1),
$$
for any $R$ small enough.
\end{remark}

In this case, it is sufficient to notice that
 \begin{align*}
 2\bns\int_{B_r}\tfrac{\epsilon_n^2}{\epsilon_n^2+|x|^2} H_ne^{\dis v_n}\,dx&=2\bns\int_{B_{\frac{r}{\epsilon_n}}}\tfrac{1}{1+|x|^2} V_ne^{\dis w_n}\,dx \\
 &\geq 2\bns\int_{B_{\rho}}\tfrac{1}{1+|x|^2} V_ne^{\dis w_n}\,dx\\
 &\geq \tfrac{2\bns}{1+\rho^2}\int_{B_{\rho}} V_ne^{\dis w_n}\,dx,
 \end{align*}
 for some $\rho>0$, which implies that
 \[
 \liminf_{n\to+\infty} 2\bns\int_{B_r}\tfrac{\epsilon_n^2}{\epsilon_n^2+|x|^2} H_ne^{\dis v_n}\,dx\geq \tfrac{16\pi}{1+\rho^2}
 \]
 since $V_ne^{\dis w_n}\rightharpoonup8\pi\delta_0$ in the sense of measure in $B_\rho$, with $\rho$ sufficiently small.
 Hence, by taking the limit $\rho\to0^+$ and then the limit $r\to0^+$, we deduce that
 \[
 \lim_{r\to0^+}\liminf_{n\to+\infty} 2\bns\int_{B_r}\tfrac{\epsilon_n^2}{\epsilon_n^2+|x|^2} H_ne^{\dis v_n}\,dx\geq 16\pi.
 \]
 Therefore, from the Pohozaev identity \eqref{Pohozaevextraterm}, we deduce that
 \[
 -\tfrac{\lm_\ii^2}{4\pi}+2(1+\bis)\lm_\ii\geq16\pi
 \]
 that is
 \[
 (\lm_\infty-8\pi)^2\leq0,
 \]
 which implies $\lm_\ii=8\pi$ and, thus, $\sigma=\tfrac{1}{2}$.\\

This fact concludes the discussion of CASE (I).\\

CASE (II): $\tfrac{\epsilon_n}{t_n}\leq C$ and $\tfrac{|x_n|}{\delta_n}\to+\infty$.\\

In particular, up to a subsequence,
\[
 \frac{\epsilon_n}{|x_n|}\to\bar \epsilon_0\geq  0,\quad \frac{x_n}{|x_n|}\to \bar z_0,\;\;|\bar z_0|=1,\quad \text{as}\,\,n\to+\infty.
\]

Let us assume that $\sigma\in(\frac{1}{4},\tfrac{1}{2}]$.

\begin{remark}\label{rem:reg2} Clearly the concentration compactness theory (\cite{bm}, \cite{yy})
for regular Liouville type equations can be applied to $\bar w_n$ in a small neighborhood of $\bar z_0$.
Actually, for any compact set $U\subset \R^2$
far away from the blow up set of $\bar w_n$, since $v_n$ satisfies \eqref{oscillation}, by Lemma 2.2 in \cite{bclt} we have that,
\begin{equation}\label{osc:barwnU}
  \underset{U}\max\,\bar w_n-\underset{U}\min\,\bar w_n\leq C_U.
\end{equation}
Therefore, in particular
\begin{equation*}
  \underset{\p B_{R}(\bar z_0)}\max\,\bar w_n-\underset{\p B_{R}(\bar z_0)}\min\,\bar w_n\leq C_{R},
\end{equation*}
for any $R$ small enough and, by the result in \cite{yy}, $\bar z_0$ is a simple blow up point for $\bar w_n$,
$$
\bar V_n e^{\dis \bar w_n}\rightharpoonup 8\pi \dt_{p=\bar z_0}\qquad\text{weakly in the sense of measures of}\,\, B_R,
$$
and
$$
\int\limits_{B_{R}(\bar z_0)}\bar V_n e^{\dis \bar w_n}=8\pi +o(1),
$$
for any $R$ small enough. Moreover, since $\sg\in(\tfrac{1}{4},\tfrac{1}{2}]$, $\lm_\ii<16\pi$. By the Minimal Mass Lemma, $\bar z_0$ is the only blow up point for $\bar w_n$. Hence, by the Concentration-Compactness alternative in Theorem \ref{Concentration-Compactness alternative}, applied to the rescaled sequence, the mass on any fixed compact subset of $\mathbb{R}^2\backslash\{\bar  z_0\}$ vanishes. Consequently, 
\[
\int\limits_{B_{R}(\bar z_0)}\bar V_n e^{\dis \bar w_n}=8\pi +o(1),
\]
for every fixed $R>0$.
\end{remark}

We split the integral in the following way: let us consider $R>0$ sufficiently large, then
\[
2\bns\int_{B_r}\tfrac{\epsilon_n^2}{\epsilon_n^2+|x|^2} H_ne^{\dis v_n}\,dx=2\bns\int_{B_{R|x_n|}(|x_n|\bar z_0)}\tfrac{\epsilon_n^2}{\epsilon_n^2+|x|^2} H_ne^{\dis v_n}\,dx+  \bar J_{n,r,R},
 \]
where
\[
\bar J_{n,r,R}:=2\bns\int_{B_r\backslash B_{R|x_n|}(|x_n|\bar z_0)}\tfrac{\epsilon_n^2}{\epsilon_n^2+|x|^2} H_ne^{\dis v_n}\,dx.
\]
 Concerning the first integral 
 \[
2\bns\int_{B_{R|x_n|}(|x_n|\bar z_0)}\tfrac{\epsilon_n^2}{\epsilon_n^2+|x|^2} H_ne^{\dis v_n}\,dx=2\bns\int_{B_R(\bar z_0)}\tfrac{\frac{\epsilon_n^2}{|x_n|^2}}{\frac{\epsilon_n^2}{|x_n|^2}+|x|^2} \bar V_ne^{\dis \bar w_n}\,dx=\bis16\pi\tfrac{\bar\epsilon_0^2}{\bar\epsilon_0^2+1}+o(1),
 \]
 where we used that $\bar V_ne^{\dis \bar w_n}\rightharpoonup8\pi\delta_{\bar z_0}$ in the sense of measure in $B_R(\bar z_0)$, while
 \[
 0\leq \bar J_{n,r,R}\leq\tfrac{2\bns C}{(R-1)^2}\lm_n.
 \]
 Moreover, we notice that
 \begin{align*}
 2\bns\int_{B_r}\tfrac{\epsilon_n^2}{\epsilon_n^2+|x|^2} H_ne^{\dis v_n}\,dx&=2\bns\int_{B_{\frac{r}{|x_n|}}}\tfrac{\frac{\epsilon_n^2}{|x_n|^2}}{\frac{\epsilon_n^2}{|x_n|^2}+|x|^2} \bar V_ne^{\dis \bar w_n}\,dx \\
 &\geq 2\bns\int_{B_{\rho}(\bar z_0)}\tfrac{\frac{\epsilon_n^2}{|x_n|^2}}{\frac{\epsilon_n^2}{|x_n|^2}+|x|^2} \bar V_ne^{\dis \bar w_n}\,dx\\
 &\geq 2\bns\tfrac{\frac{\epsilon_n^2}{|x_n|^2}}{\frac{\epsilon_n^2}{|x_n|^2}+(1+\rho)^2}\int_{B_{\rho}(\bar z_0)} \bar V_ne^{\dis \bar w_n}\,dx,
 \end{align*}
 for some $\rho>0$ sufficiently small, which implies that
 \[
 \liminf_{n\to+\infty} 2\bns\int_{B_r}\tfrac{\epsilon_n^2}{\epsilon_n^2+|x|^2} H_ne^{\dis v_n}\,dx\geq 16\pi\tfrac{\bar\epsilon_0^2}{\bar\epsilon_0^2+(1+\rho)^2}.
 \]
Hence, by taking the limit $\rho\to0^+$ and then the limit $r\to0^+$, we deduce that
 \[
 \lim_{r\to0^+}\liminf_{n\to+\infty} 2\bns\int_{B_r}\tfrac{\epsilon_n^2}{\epsilon_n^2+|x|^2} H_ne^{\dis v_n}\,dx\geq 16\pi\tfrac{\bar\epsilon_0^2}{\bar\epsilon_0^2+1}.
 \]
 Therefore, similarly as the previous case, we obtain
 \begin{align*}
16\pi\tfrac{\bar\epsilon_0^2}{\bar\epsilon_0^2+1}&\leq\lim_{r\to0^+}\liminf_{n\to+\infty}2\bns\int_{B_{r}}\tfrac{\epsilon_n^2}{\epsilon_n^2+|x|^2} H_ne^{\dis v_n}\\
&\leq\lim_{r\to0^+}\limsup_{n\to+\infty}2\bns\int_{B_{r}}\tfrac{\epsilon_n^2}{\epsilon_n^2+|x|^2} H_ne^{\dis v_n}\\
&\leq16\pi\tfrac{\bar\epsilon_0^2}{\bar\epsilon_0^2+1}+\tfrac{2\bis C\lm_\infty}{(R-1)^2}
 \end{align*}
and, taking $R\to+\infty$, we have that 
 \[
\lim_{r\to0^+}\lim_{n\to+\infty}2\bns\int_{B_{r}}\tfrac{\epsilon_n^2}{\epsilon_n^2+|x|^2} H_ne^{\dis v_n}\,dx=16\pi\tfrac{\bar\epsilon_0^2}{\bar\epsilon_0^2+1}.
 \]
 On the other hand, from the Pohozaev identity \eqref{Pohozaevextraterm}, we deduce that
 \[
 -\tfrac{\lm_\ii^2}{4\pi}+2(1+\bis)\lm_\ii=16\pi\tfrac{\bar\epsilon_0^2}{\bar\epsilon_0^2+1},
 \]
 that is
 \[
 (\lm_\infty-8\pi)^2=\frac{64\pi^2}{1+\bar\epsilon_0^2}, 
 \]
which in turn implies, since $\lm_\ii\geq8\pi$, that
$$\,\,\,\lm_\ii=8\pi\left(1+\frac{1}{\sqrt{1+\bar\epsilon_0^2}}\right)\,\,\,
 \text{and} \,\,\,\sg=\frac{1}{2}\left(1-{\frac{1}{1+\sqrt{1+\bar\epsilon_0^2}}}\right).$$
 
 \smallskip
 
Now, we consider $\sg=\tfrac14$.\\

Consequently, $\lm_\infty=16\pi$ and from \eqref{Pohozaevextraterm}, we deduce that
\[
2\bns\int_{B_r}\tfrac{\epsilon_n^2}{\epsilon_n^2+|x|^2} H_ne^{\dis v_n}\,dx=o(1).
\]
However, for $R$ small enough, we have that
\begin{align}\nonumber
2\bns\int_{B_r}\tfrac{\epsilon_n^2}{\epsilon_n^2+|x|^2} H_ne^{\dis v_n}\,dx&\geq2\bns\int_{B_{R|x_n|}(|x_n|\bar z_0)}\tfrac{\epsilon_n^2}{\epsilon_n^2+|x|^2} H_ne^{v_n}\,dx\\
    &=2\bns\int_{B_{R}(\bar z_0)}\tfrac{\frac{\epsilon_n^2}{|x_n|^2}}{\frac{\epsilon_n^2}{|x_n|^2}+|x|^2} \bar V_ne^{\bar w_n}\,dx=2\bis \frac{\bar\epsilon_0^2}{\bar\epsilon_0^2+1}8\pi+o(1).\label{argument16pi2}
\end{align}
where, again, we used that $\bar V_ne^{\dis \bar w_n}\to8\pi\delta_{\bar z_0}$ in the sense of measure in $B_R(\bar z_0)$.
This fact necessarily implies that $\bar\epsilon_0=0$. Therefore, $\lm_\infty$ and $\sg$ satisfy the formula in this case as well. This concludes the discussion of CASE (II).

\finedim

\bigskip
\bigskip

\section{A refinement of Theorem \ref{quant:imp}: mass-quantization in CASE (II)}\label{sec4}

In this section we refine the analysis of CASE (II) in Theorem \ref{quant:imp}, showing that only $\sg=\frac14$ and consequently $\lm_\ii=16\pi$
is allowed in that case. To this aim, we need to rule out the case where $\bar\epsilon_0>0$.

\begin{proposition}\label{caseII.quant}
Under the assumptions of Theorem \ref{quant:imp}, suppose that the conditions of CASE (II) in Theorem \ref{profile-new} are satisfied. Then
$\sg=\frac14$ and $\lm_\ii=16\pi$.
\end{proposition}
\begin{proof}
Clearly, by the statement of Theorem \ref{quant:imp}, it is enough to rule out the case $\bar\epsilon_0>0$. Thus,
we can assume without loss of generality that $\frac14<\sg<\frac12$ and $8\pi<\lm_\ii=\frac{4\pi}{\sg}< 16\pi$.\\
Recalling the definitions \eqref{dn1}, \eqref{barwn:IB}, \eqref{eqwn:IB}, 
we prove the following estimate
\begin{equation}
\label{estimate:vn+log-secondcase}
\underset{B_1(0)\backslash B_{4|x_n|}(0)}\sup \{v_n+2(1+\bns)\log|y|\}\leq C.
\end{equation}

Let us assume by contradiction that there exists $\{y_{n,2}\}$ such that $v_n(y_{n,2})+2(1+\bns)\log|y_{n,2}|\to+\infty$ and $4|x_n|\leq |y_{n,2}|\leq 1$. Then, we define
\[
\tilde v_{n,2}(z)=v_n(|y_{n,2}|z)+2(1+\bns)\log|y_{n,2}|,
\]
which satisfies
\begin{align}\nonumber
\begin{cases}
-\D \tilde v_{n,2}=\tilde V_{n,2}(z)e^{\dis \tilde v_{n,2}} \qquad\text{in}\,\,B_{\frac{1}{|y_{n,2}|}}(0),\\
\tilde V_{n,2}(z)=\left({\frac{\epsilon_n^2}{|y_{n,2}|^2}+|z|^2}\right)^{\bns}
K_n(|y_{n,2}|z)\to c_0(\epsilon_2^2+|z|^2)^{\bis}\;\mbox{in}\;C^r_{loc}(\R^2),\\
\int_{B_{\frac{1}{|y_{n,2}|}}}\tilde V_{n,2} e^{\dis \tilde v_{n,2}}= \lm_\ii +o(1),\\
\tilde v_{n,2}(\tfrac{y_{n,2}}{|y_{n,2}|})\to +\infty.
\end{cases}
\end{align}
where we have used the fact that $\tfrac{\epsilon_n}{|y_{n,2}|}=\tfrac{\epsilon_n}{|x_n|}\tfrac{|x_n|}{|y_{n,2}|}\leq\tfrac{C}{4}$ and $\tfrac{\epsilon_n}{|y_{n,2}|}\to\epsilon_2$, up to a subsequence. Also $\tfrac{y_{n,2}}{|y_{n,2}|}\to y_2$, with $|y_2|=1$ and, by the same argument of Remark \ref{rem:reg}, we have that $y_2$ is a simple blow up and
\[
\int_{B_{\delta|y_{n,2}|}(|y_{n,2}|y_2)}H_n e^{v_n}=\int_{B_\delta(y_2)}\tilde V_{n,2}e^{\tilde v_{n,2}}=8\pi+o(1),
\]
for $\delta>0$ small enough. We already know that
\[
\int_{B_{\rho|x_n|}(|x_n|\bar z_0)}H_ne^{v_n}=\int_{B_\rho(\bar z_0)} \bar V_ne^{ \bar w_n}=8\pi +o(1).
\]
for $\rho>0$ small enough. Moreover, $B_{\delta|y_{n,2}|}(|y_{n,2}|y_2)\cap B_{\rho|x_n|}(|x_n|\bar z_0)=\emptyset$, for $\delta$ and $\rho$ small enough.

Therefore,
\[
\int_{B_1(0)}H_n e^{\dis v_n}\geq\int_{B_{\rho|x_n|}(|x_n|\bar z_0)}H_ne^{v_n}+ \int_{B_{\delta|y_{n,2}|}(|y_{n,2}|y_2)}H_n e^{\dis v_n}= 16\pi+o(1),
\]
for $\delta$ sufficiently small, but this is in contradiction with the fact that $\int_{B_1}H_ne^{v_n}= \lm_\ii +o(1)<16\pi+o(1)$. This proves \eqref{estimate:vn+log-secondcase}.\\

Let $\tau\in(0,1)$ be defined by the equality,
$$
\lm_\ii=8\pi(1+\tau),
$$
and then define $d_n$ as follows,
$$
\int_{B_{d_n}}\bar V_n e^{\bar w_n}\,dz =  8\pi +7\pi\tau.
$$
Clearly, since for any $R\geq 2$ we have
$$
\int_{B_{R}}\bar V_n e^{\bar w_n}\,dz= 8\pi+o(1),
$$
then
$$
d_n\to +\ii.
$$
On the other side, since $H_ne^{\dis v_n}\rightharpoonup \lm_\ii\dt_{p=0}$, then
$$
|x_n|d_n\to 0.
$$
Remark in particular that by construction we have,
\begin{align}\label{mass:II}
\int_{B_{d_n}^c}\bar V_n e^{\bar w_n}\,dz= \pi \tau +o(1).
\end{align}
Let
$$
w_n^*(y)=\bar w_n(d_ny)+2(1+\bns)\log(d_n), \quad |y|\leq \frac{r}{2|x_n|d_n},
$$
which satisfies
\begin{align}\label{eqwn:IB*}
\begin{cases}
-\D  w_{n}^*= V^*_{n}(y)e^{\dis w_{n}^*} \quad\text{in}\,\, |y|\leq \frac{r}{2|x_n|d_n},\\
V^*_{n}(y)=\left({\frac{\epsilon_n^2}{|x_{n}|^2d_n^2}+|y|^2}\right)^{\bns}K_n(|x_{n}|d_n y)\to
K(0)|y|^{2\bis}\;\mbox{in}\;C^t_{loc}(\R^2),\\
\int_{B_{\frac{r}{2|x_n|d_n}}}V^*_n e^{w_n^*}\,dz=\lm_\ii+o(1), \,\,\,\,\int_{B_{1}}V^*_n e^{w_n^*}\,dz= 8\pi+7\pi\tau+o(1), \\
\int_{B_{\frac{R}{d_n}}}V^*_n e^{w_n^*}\,dz= 8\pi+o(1),\quad \int_{B_{1}^c}V^*_n e^{w_n^*}\,dz= \pi\tau +o(1),\\
w_{n}^*(\tfrac{x_n}{|x_n|d_n})=\bar w_{n}(\tfrac{x_n}{|x_n|})+2(1+\bns)\log(d_n)\to +\infty,\qquad \tfrac{x_n}{|x_n|d_n}\to0.
\end{cases}
\end{align}
We show that, since $\tfrac{\epsilon_n}{|x_n|d_n}\leq\tfrac{C}{d_n}\to0$, then $w_n^*$ concentrates in the sense
of Theorem \ref{Concentration-Compactness alternative}. Observe first that the argument in the (Minimal Mass) Lemma \ref{minimalmasslemma} works as well
for $w_n^*$. Indeed, defining $\beta^*$ as follows,
\[
\beta^*=\lim_{\delta\to0}\,\liminf_{n\to+\infty}\int_{B_\delta}V_n^*e^{w_n^*}
\]
then we can prove that $\beta^*\geq 8\pi$. Moreover, since $\bis=1$ and $\beta^*\geq8\pi$, we can apply Theorem 2.1 in
\cite{os} (see also Remark \ref{rem:concentrationgeneral}) and deduce that $w_n^*\to-\infty$ uniformly on compact sets
far away from the blow up point $y=0$.
Furthermore, by Lemma 2.2 in \cite{bclt}, $w_n^*$ satisfies
\[
\sup_{\p B_R} w_n^*-\inf_{\p B_R} w_n^*\leq C_R.
\]
At this point, we prove that
\begin{equation}
    \label{wn*+logbounded}
    w_n^*(y)+2(1+\bns)\log|y|\leq C, \quad \forall |y|\geq 2.
\end{equation}
Indeed, since \eqref{estimate:vn+log-secondcase} holds, it follows that 
\[
\bar w_n(z)+2(1+\bns)\log|z|\leq C, \quad \forall |z|\geq 4
\]
and taking $z=d_ny$, for $n$ large enough, we obtain the estimate for any $|y|\geq2$.
Hence, we have proved that \eqref{wn*+logbounded} holds and then, in particular, we have
\begin{equation*}
V^*_ne^{\dis w_{n}^*}\rightharpoonup (8\pi+7\pi\tau)\dt_{p=0}, \quad \mbox{in}\quad B_1,
\end{equation*}
and
\begin{equation}\label{conc:wnstar}
V^*_ne^{\dis w_{n}^*}\rightharpoonup (8\pi+7\pi\tau)\dt_{p=0},\quad \mbox{in}\quad B_R,
\end{equation}
for any $R\geq 1$. We will need the following estimates.
\ble\label{lem:barwn}
Let us define,
$$
\omega_n=\tfrac{1}{2\pi}\int_{B_{1}}\log({|\xi-\xi_n|})V^*_n e^{w_n^*}\,d\xi,
$$
where $\xi_n=\tfrac{z_0}{d_n}$, for some $|z_0|=2$. Then we have
$$
w_n^*(y)\leq  \omega_n-2(1+\bns)\log\left(\frac{\bar \theta_n}{d_n}\right)+(4+\tau)\log\frac{1}{|y|}+ O(1), \quad \forall\, |y|\geq 2;
$$
$$
e^{\omega_n}\left(\frac{d_n}{\bar \theta_n}\right)^{2(1+\bns)}\leq C.
$$
\ele
For ease of the presentation, we postpone the proof of Lemma \ref{lem:barwn} to subsection \ref{sec:barwnstar} below. By using the Lemma we can conclude the proof.
In fact, by using \eqref{conc:wnstar} and the Lemma we have, for any $R\geq2$,
$$
\pi \tau+o(1)=\int\limits_{1\leq |\xi|\leq R}V^*_{n}e^{\dis w_{n}^*}\,d\xi + \int\limits_{|\xi|\geq R}V^*_{n}e^{\dis w_{n}^*}\,d\xi=
o(1)+\int\limits_{|\xi|\geq R}V^*_{n}e^{\dis w_{n}^*}\,d\xi\leq
$$
$$
 o(1)+Ce^{\omega_n}\left(\frac{d_n}{\bar \theta_n}\right)^{2(1+\bns)}
\int\limits_{|\xi|\geq R}\frac{|\xi|^{2\bns}}{|\xi|^{4+\tau}}\,d\xi\leq
o(1)+Ce^{\omega_n}\left(\frac{d_n}{\bar \theta_n}\right)^{2(1+\bns)}\int\limits_{R}^{+\ii}\frac{dt}{t^{1+\tau}}=o(1)+\frac{C}{R^\tau},
$$
which is the desired contradiction.
\end{proof}

\bigskip

\section{Sharp Harnack-type inequalities}
\label{sectHarnackineq}
In this section, we establish sharp Harnack-type inequalities.

\begin{theorem}\label{supinfI}
Under the hypothesis of Theorem \ref{massboundarycontrol},
assume furthermore that 
$$
\bns=\frac{\lm_n}{4\pi} \sg\leq 1.
$$
If the condition of CASE (I) in Theorem \ref{profile-new} is satisfied, that is $\tfrac{\epsilon_n}{t_n}\to+\infty$,
then, there exists $C_I=C_I(a,A,\|K\|_\ii,\L)>0$ such that,
\begin{equation}\label{sup+inf-I}
v_n(x_n)+\inf\limits_{B_1} v_n\leq C_I -2(1+\bns)\log\epsilon_n+2(1-\bns)\log(\epsilon_n).
\end{equation}

If the condition of CASE (II) in Theorem \ref{profile-new} is satisfied,
that is, there exists a subsequence such that $\tfrac{\epsilon_n}{t_n}\leq C$ and $\tfrac{|x_n|}{\delta_n}\to+\infty$,
then, there exists $C_{II}=C_{II}(a,A,\|K\|_\ii,\L)>0$ such that,
\begin{equation}\label{sup+inf-II}
v_n(x_n)+\inf\limits_{B_1} v_n\leq C_{II} -2(1+\bns)\log|x_n|+2(1-\bns)\log(|x_n|).
\end{equation}
If the condition of CASE (III) in Theorem \ref{profile-new} is satisfied,
that is, there exists a subsequence such that $\tfrac{\epsilon_n}{t_n}\leq C$ and $\tfrac{|x_n|}{\delta_n}\leq \tilde C$, then, there exists $C_{III}=C_{III}(a,A,\|K\|_\ii,\L)>0$ such that,
\begin{equation}\label{sup+inf-III}
v_n(x_n)+\inf\limits_{B_1} v_n\leq C_{III}-2(1+\bns)\log\delta_n+2(1-\bns)\log\delta_n.
\end{equation}

\end{theorem}
\begin{remark} We emphasize that the constants $C_{I},C_{II}$ and $C_{III}$ do not depend on the particular solution in the given case. Remark that \eqref{sup+inf-I} is sharp in the sense that if we take for example the profile \eqref{profilev:H} in Theorem
\ref{profile-new} above, then we find that,
$$
v_n(x_n)+\inf\limits_{B_1} v_n= O(1)-2(1+\bns)\log\epsilon_n+2(1-\bns+o(1))\log(\epsilon_n),
$$
where $o(1)$ is just the difference $\frac{\lm_n}{4\pi}-2$. The same property holds true about \eqref{sup+inf-II}.
On the other side, if along a subsequence we had $\bns=1$, then \eqref{sup+inf-III} would reduce to $\inf\limits_{B_1} v_n\leq C_{III}$,
which is a straightforward consequence of \eqref{lambdan}. Far from being a technical point, this fact reflects a subtle property
of the blow up phenomenon encoded by \eqref{eqbase}: as far as $\bns>1$,
it is not true in general that blow up implies concentration (\cite{det,llty,os}). See also Corollary \ref{cor:supinf}.
\end{remark}

\proof
Clearly, in view of \eqref{bounded} and \eqref{explosion}, we can assume that \eqref{blowup:hyp} holds true.
Remark that by assumption (Theorem \ref{massboundarycontrol}) we have $\sg\in\left(0,\tfrac12\right]$ and by \eqref{quantized:poslm} we have that
\begin{equation}\label{lm:bound}
\lm_\ii\leq \min\left\{\frac{4\pi}{\sg},\frac{8\pi}{1-2\sigma}\right\}\leq 16 \pi.
\end{equation}
We first discuss CASE (I).\\
Recall that, in view of $\tfrac{\epsilon_n}{t_n}\to+\infty$, \eqref{epsilondelta&epsilonxn} is satisfied
(that is $\frac{\epsilon_n}{\delta_n}\to+\infty$ and $\frac{\epsilon_n}{|x_n|}\to+\infty$)
and we define $B^{(n)}$ and $w_n$ as in \eqref{bn1}, \eqref{wn:IA}. Therefore $w_n$ is a solution of \eqref{eqwn:IA2} and $z=0$
is a blow up point for $w_n$ in $B_1$.

In view of a standard blow up argument (\cite{ls}),
putting $ \theta_{n}^{2(1+\al_{n,\sg})}=
\left(\tfrac{\epsilon_n}{\delta_{n}}\right)^{2(1+\al_{n,\sg})}\to +\ii$, $z_n=\tfrac{x_n}{\epsilon_n}$ 
and $8\gamma_{n}=V_{n}(\tfrac{x_n}{\epsilon_n})$, we also have that, for any $R\geq 1$,
\begin{equation}\label{yy-convI.00}
f_n(z):=\left|w_n(z)-\log\left(\dfrac{\theta_{n}^{2(1+\al_{n,\sg})}}
{\left(1+\gamma_{n}\theta_{n}^{2(1+\al_{n,\sg})}|z-z_n|^2\right)^2}\right) \right|,
\end{equation}
satisfies
\begin{equation}\label{yy-convI}
\max\left\{f_n(z)\,|\,|z-z_n|^2\leq R^2\theta_{n}^{-2(1+\al_{n,\sg})}\right\}\to 0.
\end{equation}
Actually, by \cite{yy}, we have
\begin{equation}\label{yy-convI.0}
\left|w_n(z)-\log\left(\dfrac{\theta_{n}^{2(1+\al_{n,\sg})}}
{\left(1+\gamma_{n}\theta_{n}^{2(1+\al_{n,\sg})}|z-z_{n}|^2\right)^2}\right) \right|\leq C, \quad \forall |z|\leq R,
\end{equation}
for any $R>0$ small enough.

At this point we argue as in \cite{bls} and define,
$$ Q=\{(t,\phi):\;t\leq 0,\;\phi\in[0,2\pi]\},$$
\begin{equation}\label{def:un}
u_n(t,\phi) = v_n(e^{t}\cos{\phi},
e^{t}\sin{\phi})+2t-\frac{A}{a}e^{t},\;\;(t,\phi)\in Q.
\end{equation}
Thus $u_n$ satisfies:
\begin{equation}\label{eq:un}
-\D u_n= {W}_n (t,\phi) e^{\dis u_n}
+\frac{A}{a}e^{t}\;\;\mbox{in}\;\; Q,
\end{equation}
where
\begin{equation}\label{def:Wn}
 W_n (t,\phi)=
(\epsilon_n^2+e^{2t})^{\bns}e^{\frac{A}{a}e^{t}} \hat K_{n}(t,\phi),\quad \hat K_{n}(t,\phi):= K_n(e^{t}\cos{\phi},
e^{t}\sin{\phi}),
\end{equation}
with $a$ and $A$ defined respectively in \eqref{defa} and \eqref{gradientbound}. By using that $\left|\frac{\pa}{\pa t} \hat K_{n}(t,\phi)\right|\leq Ae^t$,
it is readily seen that
\begin{equation}\label{monotone-V}
\frac{\pa}{\pa t}\left( W_n (t,\phi)e^\xi+\frac{A}{a}e^{t}\right)\geq \frac{A}{a}e^{t}>0,\quad \forall \xi\in \R.
\end{equation}
Since for $n$ fixed we have,
\begin{equation}\nonumber
u_n(t,\phi) = 2t+a_n+\mbox{O}_n(e^{t}), \quad t\to -\ii,
\end{equation}
\begin{equation}\nonumber
\frac{\pa}{\pa t} u_n(t,\phi) = 2+\mbox{O}_n(e^{t}), \quad t\to -\ii,
\end{equation}
then, there exists $\mu<0$ depending on $n$ such that, for all $\nu\leq \mu$ we have
$$
u_n(2\nu-t,\phi)-u_n(t,\phi)=4(\nu-t)+c_n+\mbox{O}_n(e^{t})\leq \nu+c_n ,\quad \forall\,\frac{\nu}{2}\leq t\leq 0,\;\phi\in [0,2\pi],
$$
and
\begin{equation}\nonumber
\frac{\pa}{\pa t} u_n(t,\phi) = 2+\mbox{O}_n(e^{t})\geq 2-c_ne^{\frac{\nu}{2}}, \quad \forall\,  t\leq \frac{\nu}{2},\;\phi\in [0,2\pi].
\end{equation}
Therefore there exists $\mu$ depending by $n$ such that, for all $\nu\leq \mu$ we have
$$
u_n(2\nu-t,\phi)-u_n(t,\phi)<0,\quad \forall\,\nu<t<0,\;\phi\in [0,2\pi],
$$
and consequently is well defined
$$
\mu_n=\sup\{\mu<0\,:\,\forall \nu< \mu,\; u_n(2\nu-t,\phi)-u_n(t,\phi)<0,\quad\forall\,\nu<t<0,\;\phi\in [0,2\pi]\}.
$$

We claim that
\begin{equation}\label{mp-comp}
\min\limits_{\phi \in [0,2\pi)}u_n(0,\phi)\leq \max\limits_{\phi \in [0,2\pi)}u_n(2\mu_n,\phi).
\end{equation}

Let $\varphi_n(t,\theta)=u_n(2\mu_n-t,\phi)-u_n(t,\phi)$,
then, by using also \eqref{monotone-V}, we have $\varphi_n\leq 0$ and $\Delta \varphi_n\geq 0$ for $(t,\phi)\in [\mu_n,0]\times [0,2\pi)$. Since
 $\varphi_n(\mu_n,\phi)\equiv 0$ and in view of the maximality of $\mu_n$, by the strong maximum principle we must have $\varphi_n(0,\phi)=0$
 for some $\phi \in [0,2\pi)$, which implies \eqref{mp-comp}.\\
 At this point we observe that, in view of \eqref{yy-convI}, for any $R\geq 1$ we have
 \begin{equation}\label{yy-convI1}
 w_n(z_n+s_ny)=w_n(z_n)+\log\frac{1}{(1+\gamma_n|y|^2)^2}+o(1),\quad |y|\leq R,
 \end{equation}
 where $s_{n}:= \theta_n^{-(1+\bns)}\to0^+$ and, passing to complex notations, we set,
$$
 \phi_n=\arg(z_{n})=\arg(x_{n}), \qquad z_n=|z_{n}|e^{i\phi_n}.
$$

In \eqref{yy-convI1} we chose $|y_{n,1}|=R_1$ and $|y_{n,2}|=R_2$ for some $R_1\geq 2$ and $R_2\geq 0$ to be fixed later on such that,
$$
 \arg(y_{n,1})=\phi_n=\arg(y_{n,2}), \quad y_{n,1}=|y_{n,1}|e^{i\phi_n},\quad  y_{n,2}=|y_{n,2}|e^{i\phi_n}.
$$

At this point we observe that, for $i=1,2$,
\begin{align*}
 &w_n(z_n+s_ny_{n,i})=v_n(\epsilon_n(z_n+s_ny_{n,i}))+2(1+\bns)\log(\eps_n)=\\
 & u_n(\log(\eps_n)+\log(|z_n|+s_n|y_{n,i}|),\phi_n)-2\log(\eps_n)-2\log(|z_n|+s_n|y_{n,i}|)+\\
 &\frac{A}{a}\eps_n(|z_n|+s_n|y_{n,i}|)+2(1+\bns)\log(\eps_n),
 \end{align*}
and consequently, by using \eqref{yy-convI1},
\begin{align*}
    &u_n(\log(\eps_n)+\log(|z_n|+s_n|y_{n,i}|),\phi_n)-2\log(|z_n|+s_n|y_{n,i}|)+2 \log(1+\gamma_n R_i^2) \\
    =&u_n(\log(\eps_n)+\log(|z_n|),\phi_n)-2\log(|z_n|)+o(1).
\end{align*}

Therefore, we find that,
\begin{align*}
    &u_n(\log(\eps_n)+\log(|z_n|+s_n|y_{n,1}|),\phi_n)-2\log(|z_n|+s_n R_1)+2 \log(1+\gamma_n R_1^2)\\
    =&u_n(\log(\eps_n)+\log(|z_n|+s_n|y_{n,2}|),\phi_n)-2\log(|z_n|+s_n R_2)+2 \log(1+\gamma_n R_2^2)+o(1),
\end{align*}
that is
\begin{align*}
    &u_n(\log(\eps_n)+\log(|z_n|+s_n|y_{n,1}|),\phi_n)-u_n(\log(\eps_n)+\log(|z_n|+s_n|y_{n,2}|),\phi_n)\\
    =&2\log\frac{|z_n|+s_n R_1}{|z_n|+s_n R_2}+2 \log\frac{1+\gamma_n R_2^2}{1+\gamma_n R_1^2}+o(1).
\end{align*}

We split the discussion in two cases: the first where $s_n/|z_n|$ is bounded, and the second in which $s_n/|z_n|$ tends to infinity.\\
CASE (I-1): $\frac{s_n}{|z_n|}\leq 1$  for $n$ large enough .\\
In this case we choose $R_2=0$ and $R_1$ large enough so that,
\begin{align*}
    &2\log\frac{|z_n|+s_n R_1}{|z_n|+s_n R_2}+2 \log\frac{1+\gamma_n R_2^2}{1+\gamma_n R_1^2}+o(1)\\
    =&2 \log\frac{1+\frac{s_n}{|z_n|}R_1}{1+\gamma_n R_1^2}+o(1)\leq  2 \log\frac{1+ R_1}{1+\gamma_n R_1^2}+o(1)<0.
\end{align*}
Therefore, setting $t_n=\log(\eps_n)+\log(|z_n|+s_n|y_{n,1}|)$ and
$$
\nu_n:=\log(\eps_n)+\log(|z_n|)+\frac12\log(1+\tfrac{s_n}{|z_n|}|y_{n,1}|),
$$
we have $\nu_n<t_n$ and
$$
u_n(t_n,\phi_n)<u_n(2\nu_n-t_n, \phi_n),
$$
which shows that,
\begin{align*}
\mu_n\leq \nu_n\leq \log(|x_n|)+\frac12 \log(1+\tfrac{s_n}{|z_n|}R_1)\leq \log(|x_n|)+ C,
\end{align*}
for a suitable large positive constant $C_L$. Therefore, since
$$
1\leq \left(\frac{|z_n|}{s_n}\right)^2=\frac{\eps_{n}^{2(1+\al_{n,\sg})}}{\dt_{n}^{2(1+\al_{n,\sg})}}\frac{|x_n|^2}{\eps_n^2}=
\left(\frac{\eps_n}{\dt_n}\right)^{2\bns}\left(\frac{|x_n|}{\dt_n}\right)^2,
$$
by \eqref{yy-convI.0} and \eqref{mp-comp}, for some sequence $\xi_n\in \R^2$, $|\xi_n|\leq e^{2C}$,  we find that,
\begin{align*}
        &\inf\limits_{\pa B_1} v_n - \tfrac{A}{a} = \min\limits_{\phi\in [0,2\pi]} u_n(0,\phi)\leq \max\limits_{\phi\in [0,2\pi]} u_n(2\mu_n,\phi) \\
    =&\max\limits_{\phi\in [0,2\pi]} v_n(e^{2\mu_n}\cos\phi,e^{2\mu_n}\sin\phi)+4\mu_n-\frac{A}{a}e^{2\mu_n}
\end{align*}
\begin{align*}
    \leq & v_n(|x_n|^2\xi_n)+4\log|x_n|+O(1)\leq \log\left(\dfrac{\dt_{n}^{-2(1+\al_{n,\sg})}|x_n|^4}
    {\left(1+\gamma_{n}\theta_{n}^{2(1+\al_{n,\sg})}|\frac{|x_n|^2}{\eps_n}\xi_n-\frac{x_n}{\eps_n}|^2\right)^2}\right)+O(1)  \\
    \leq  & \log\left(\dfrac{\dt_{n}^{-2(1+\al_{n,\sg})}|x_n|^4}
    {\left(\frac{\eps_{n}^{2(1+\al_{n,\sg})}}{\dt_{n}^{2(1+\al_{n,\sg})}}\frac{|x_n|^2}{\eps_n^2}|{|x_n|}\xi_n-\frac{x_n}{|x_n|}|^2\right)^2}\right)+O(1)=
    \log\left(\dfrac{\dt_{n}^{-2(1+\al_{n,\sg})}|x_n|^4}
    {\left(\frac{\eps_{n}^{2(1+\al_{n,\sg})}}{\dt_{n}^{2(1+\al_{n,\sg})}}\frac{|x_n|^2}{\eps_n^2}\right)^2}\right)+O(1)\\
    =&  \log\left(\dfrac{\dt_{n}^{-2(1+\al_{n,\sg})}|x_n|^4}
    {\left(\frac{\eps_{n}^{2\al_{n,\sg}}}{\dt_{n}^{2(1+\al_{n,\sg})}}{|x_n|^2}\right)^2}\right)+O(1)=
    \log\left(\frac{\dt_{n}^{2(1+\al_{n,\sg})}}{\eps_n^{4\bns}}\right)+O(1)  \\
    =&  2(1+\bns)\log\left(\frac{\dt_n}{\eps_n}\right)+2(1-\bns)\log\eps_n+O(1),
\end{align*}
which is \eqref{sup+inf-I}.\\

\noi CASE (I-2): $\frac{s_n}{|z_n|}> 1$ for $n$ large enough.\\
In this case we choose $R_2=1$ and $R_1=R$ for some $R\geq 2$ large enough so that,
\begin{align*}
    &2\log\frac{|z_n|+s_n R}{|z_n|+s_n}+2 \log\frac{1+\gamma_n }{1+\gamma_n R^2}+o(1)=2\log\frac{\frac{|z_n|}{s_n}+R}{\frac{|z_n|}{s_n}+1}+2 \log\frac{1+\gamma_n }{1+\gamma_n R^2}+o(1)\\
    =&2\log\frac{\frac{|z_n|}{s_n}+R}{1+\gamma_n R^2}+2 \log\frac{1+\gamma_n }{\frac{|z_n|}{s_n}+1}+o(1)\leq 2\log\frac{(1+R)(1+\gamma_n)}{1+\gamma_n R^2}+o(1)<0.
\end{align*}
Therefore, setting $t_n=\log(\eps_n)+\log(|z_n|+s_n|y_{n,1}|)$ and
$$
\nu_n:=\log(\eps_n)+\frac12\log\left((|z_n|+{s_n}|y_{n,1}|)(|z_n|+{s_n}|y_{n,2}|)\right),
$$
since $|y_{n,2}|=1<2\leq R=|y_{n,1}|$ we have $\nu_n<t_n$ and
$$
u_n(t_n,\phi_n)<u_n(2\nu_n-t_n, \phi_n),
$$
which shows that, since $\frac{|z_n|}{s_n}<1$,
\begin{align*}
\mu_n\leq \nu_n\leq \log(\eps_n) +\log(s_n)+\frac12 \log\left(\left(\tfrac{|z_n|}{s_n}+R\right)\left(\tfrac{|z_n|}{s_n}+1\right)\right)\leq \log(\eps_n s_n)+ C,
\end{align*}
for a suitable large positive constant $C$. Therefore,
by \eqref{yy-convI.0} and \eqref{mp-comp}, for some sequence $\xi_n\in \R^2$, $|\xi_n|\leq e^{2C}$,  we find that,
\begin{align*}
    &  \inf\limits_{\pa B_1} v_n -\tfrac{A}{a} = \min\limits_{\phi\in [0,2\pi]} u_n(0,\phi)\leq \max\limits_{\phi\in [0,2\pi]} u_n(2\mu_n,\phi)  \\
    =&  \max\limits_{\phi\in [0,2\pi]} v_n(e^{2\mu_n}\cos\phi,e^{2\mu_n}\sin\phi)+4\mu_n-\frac{A}{a}e^{2\mu_n}
\end{align*}
\begin{align*}
    \leq  &   v_n(\eps_n^2 s_n^2\xi_n)+4\log(\eps_n s_n)+O(1)\leq \log\left(\dfrac{\dt_{n}^{-2(1+\al_{n,\sg})}(\eps_n s_n)^4}
    {\left(1+\gamma_{n}\theta_{n}^{2(1+\al_{n,\sg})}|\frac{(\eps_n s_n)^2}{\eps_n}\xi_n-z_n|^2\right)^2}\right)+O(1)  \\
    = & \log\left(\dfrac{\dt_{n}^{-2(1+\al_{n,\sg})}(\eps_n s_n)^4}
    {\left(1+\gamma_{n}|{\eps_n}s_n\xi_n-\frac{z_n}{s_n}|^2\right)^2}\right)+O(1)\leq
    \log\left(\dt_{n}^{-2(1+\al_{n,\sg})}(\eps_n s_n)^4\right)+O(1)  \\
    =&  \log\left({\dt_{n}^{-2(1+\al_{n,\sg})}\eps_n^{-4\alpha_{n,\sigma}} \dt_{n}^{4(1+\al_{n,\sg})}}\right)+O(1)\\
    =&  2(1+\bns)\log\left(\frac{\dt_n}{\eps_n}\right)+2(1-\bns)\log\eps_n+O(1),
\end{align*}
which is \eqref{sup+inf-I}. Thus, the discussion of Case I is concluded.\\

Next we discuss CASE (II).\\
Recall that in this case
$\delta_n\leq|x_n|$, for $n$ large enough, whence \eqref{epsilonxn} is satisfied,
\[
 \frac{|x_n|}{\delta_n}\to+\infty,\,\,\text{as}\,\,n\to+\infty,
\]
and, possibly along a further subsequence,

\[
 \frac{\epsilon_n}{|x_n|}\to\bar \epsilon_0\geq  0,\quad \bar z_n=\frac{x_n}{|x_n|}\to \bar z_0,\;\;|\bar z_0|=1,\quad \text{as}\,\,n\to+\infty.
\]
In this situation we define $D^{(n)}$ as in \eqref{dn1} and $\bar w_{n}$ as in \eqref{barwn:IB} which therefore is a solution of
\eqref{eqwn:IB}. Therefore $\bar z_0$ is a blow up point for $\bar w_n$. It is readily seen that Remark
\ref{rem:reg} applies to $\bar w_n$ as well, whence, in particular for any compact set $U\subset \R^2$
far away from the blow up set of $\bar w_n$, since $v_n$ satisfies \eqref{oscillation}, by Lemma 2.2 in \cite{bclt} we have that,
\begin{equation*}
  \underset{U}\max\,\bar w_n-\underset{U}\min\,\bar w_n\leq C_U,
\end{equation*}
and in particular $\bar z_0$ is a simple blow up point and
$$
\int\limits_{B_{R}(\bar z_0)}\bar V_n e^{\dis \bar w_n}=8\pi +o(1),
$$
for any $R$ small enough.

In view of a standard blow up argument (\cite{ls}),
putting $\bar \theta_{n}^{2(1+\al_{n,\sg})}=
\left(\tfrac{|x_{n}|}{\delta_{n}}\right)^{2(1+\al_{n,\sg})}\to +\ii$ and $8\bar \gamma_{n}=\bar V_{n}(\tfrac{ x_{n}}{|x_n|})$, we also have that, for any $R\geq 1$,
\begin{equation}\label{yy-convII.0}
\bar f_n(z):=\left|\bar w_n(z)-\log\left(\dfrac{\bar \theta_{n}^{2(1+\al_{n,\sg})}}
{\left(1+\bar \gamma_{n}\bar\theta_{n}^{2(1+\al_{n,\sg})}|z- \bar z_n|^2\right)^2}\right) \right|,
\end{equation}
satisfies
\begin{equation}\label{yy-convII}
\max\left\{\bar f_n(z)\,|\,|z-\bar z_n|^2\leq R\bar \theta_{n}^{-2(1+\al_{n,\sg})}\right\}\to 0.
\end{equation}
Actually, by \cite{yy}, we have
\begin{equation}\label{yy-convII.1}
\left|\bar w_n(z)-\log\left(\dfrac{\bar\theta_{n}^{2(1+\al_{n,\sg})}}
{\left(1+\bar\gamma_{n}\bar\theta_{n}^{2(1+\al_{n,\sg})}|z-\bar z_{n}|^2\right)^2}\right) \right|\leq C, \quad \forall |z-\bar z_0|\leq R,
\end{equation}
for any $R>0$ small enough.

We observe that we are not able to extend \eqref{yy-convII.1} in a small neighborhood of $0$, since the potential is not uniformly bounded from below away from zero. Therefore, we cannot directly apply the estimate in \cite{yy}, as in the previous case. However, we are able to prove that there exists $0<\eta<1$ such that
\begin{equation}
\label{estimatearound0barwn}
    \sup_{B_{\eta}(0)}|\bar w_n+2(1+\bns)\log\tfrac{|x_n|}{\delta_n}|\leq C.
\end{equation}
Firstly, we prove that $0$ cannot be a blow up point for $\bar w_n$. The proof of this fact is postponed in the Appendix.\\
At this point, there exists $0<\eta<1$ such that $B_{\eta}(0)$ does not contain blow up points of $\bar w_n$. By using \eqref{osc:barwnU} and \eqref{yy-convII.1}, we deduce that
\begin{equation}
\label{boundaryboundaroun0}
\sup_{\p B_{\eta}(0)}|\bar w_n+2(1+\bns)\log\tfrac{|x_n|}{\delta_n}|\leq C.
\end{equation}
Now, since there are no blow up points in $B_{\eta}(0)$, $\bar w_n\leq C$ and, therefore $0\leq \bar f_n:=(\tfrac{\epsilon_n^2}{|x_n|^2}+|x|^2)^{\bns}K_n(|x_n|x)e^{\bar w_n}\leq C$. Let $p_n$ the solution of
\[
\graf{-\lapl p_n=\bar f_n\,\,&\text{in}\,\,B_\eta(0),\\ p_n=0\,\,&\text{on}\,\,\p B_\eta(0).}
\]
By comparison with the function $\tfrac C4(\eta^2-|z|^2)$, we obtain that
\begin{equation}
    \label{boundpn}
0\leq p_n\leq C\eta^2.
\end{equation}
Finally, the function $\bar w_n-p_n+2(1+\bns)\log\tfrac{|x_n|}{\delta_n}$ is harmonic and, by \eqref{boundaryboundaroun0}, is uniformly bounded in modulus on $\p B_{\eta}(0)$. Hence, by the maximum principle and \eqref{boundpn}, we prove \eqref{estimatearound0barwn}.

\medskip

At this point we argue as in CASE (I) and define:
$$ Q=\{(t,\phi):\;t\leq 0,\;\phi\in[0,2\pi]\},$$
and $u_n$ as in \eqref{def:un} which solves \eqref{eq:un}
where $W_n$ is defined in \eqref{def:Wn} and satisfies \eqref{monotone-V}. In particular the same argument adopted above
shows that it is well defined
$$
\mu_n=\sup\{\mu<0\,:\,\forall \nu< \mu,\; u_n(2\nu-t,\phi)-u_n(t,\phi)<0,\quad\forall\,\nu<t<0,\;\phi\in [0,2\pi]\},
$$
and that \eqref{mp-comp} holds true in this case as well.
At this point we observe that, in view of \eqref{yy-convII}, for any $R\geq 1$ we have
\begin{equation}\label{yy-convII1}
 \bar w_n(\bar z_n+\bar s_ny)=\bar w_n(\bar z_n)+\log\frac{1}{(1+\bar \gamma_n|y|^2)^2}+o(1),\quad |y|\leq R,
\end{equation}
where $\bar s_n:=(\bar\theta_n)^{-(1+\bns)}$ and, passing to complex notations, we set,
$$
\phi_n=\arg(\bar z_{n})=\arg(x_{n}), \qquad \bar z_n=|\bar z_{n}|e^{i\phi_n}=e^{i\phi_n}.
$$
Next, let us fix  $|\bar y_n|=R$ such that
$$
 \arg(\bar y_n)=\phi_n, \quad \bar y_n=|\bar y_n|e^{i\phi_n}.
 $$
 At this point we observe that
\begin{align*}
 &\bar w_n(\bar z_n+\bar s_n\bar y_n)=v_n(|x_n|(\bar z_n+\bar s_n\bar y_n))+2(1+\bns)\log(|x_n|)\\
 =& u_n(\log(|x_n|)+\log(|\bar z_n|+\bar s_n|\bar y_n|),\phi_n)-2\log(|x_n|)-2\log(|\bar z_n|+\bar s_n|\bar y_n|)+\\
 &\frac{A}{a}|x_n|(|\bar z_n|+\bar s_n|\bar y_n|)+2(1+\bns)\log(|x_n|),
 \end{align*}
and consequently, recalling that $|\bar z_n|=1$, by using \eqref{yy-convII1},
\begin{align*}
    &u_n(\log(|x_n|)+\log(1+\bar s_n|\bar y_n|),\phi_n)-2\log(1+\bar s_n|\bar y_n|)\\
    =&u_n(\log(|x_n|)+\log(1),\phi_n)-2\log(1)-2 \log(1+\bar\gamma_n R^2)+o(1),
\end{align*}
that is,
$$
u_n(\log(|x_n|)+\log(1+\bar s_n|\bar y_n|),\phi_n)= u_n(\log(|x_n|),\phi_n)+2\log\left(\frac{1+\bar s_nR}{1+\bar\gamma_n R^2}\right)+o(1).
$$
Therefore, since $\bar s_n\to 0$, for any $R$ large enough,
setting $\bar t_n=\log(|x_n|)+\log(1+\bar s_n|\bar y_n|)$ and
$$
\nu_n:=\log(|x_n|)+\frac12\log(1+\bar s_n|\bar y_n|),
$$
we have that,
$$
u_n(\bar t_n,\phi_n)<u_n(2\nu_n-\bar t_n, \phi_n),
$$
which shows that
$$
\mu_n\leq \nu_n=\log(|x_n|)+\frac12\log(1+\bar s_nR)\leq \log(|x_n|)+C,
$$
for a suitable positive constant $C$. Therefore,
by \eqref{mp-comp} and \eqref{estimatearound0barwn}, for some sequence $\xi_n\in \R^2$, $|\xi_n|\leq e^{2C}$,  we find that,
\begin{align*}
    &\inf\limits_{\pa B_1} v_n -\tfrac{A}{a}= \min\limits_{\phi\in [0,2\pi]} u_n(0,\phi)\leq \max\limits_{\phi\in [0,2\pi]} u_n(2\mu_n,\phi)\\
    =&\max\limits_{\phi\in [0,2\pi]} v_n(e^{2\mu_n}\cos\phi,e^{2\mu_n}\sin\phi)+4\mu_n-\frac{A}{a}e^{2\mu_n}
\end{align*}
\begin{align*}
    \leq & v_n(|x_n|^2\xi_n)+4\log|x_n|+O(1)= \bar w_n(|x_n|\xi_n)-2(1+\bns)\log|x_n|+4\log|x_n|+O(1)  \\
    \leq &-2(1+\bns)\log(\tfrac{|x_n|}{\delta_n})-2(1+\bns)\log|x_n|+4\log|x_n|+O(1)=
    \\
    =&  2(1+\bns)\log\left(\frac{\dt_n}{|x_n|}\right)+2(1-\bns)\log|x_n|+O(1),
\end{align*}
which is \eqref{sup+inf-II}.
This fact concludes the discussion of CASE (II).

\bigskip

Next we discuss CASE (III). Recall that in this case, possibly along a subsequence, we assume without loss of generality that,
\begin{equation*}
 \frac{x_n}{\delta_n}\to y_0, \,\,\text{as}\,\,n\to+\infty
\end{equation*}
 and
\begin{equation*}
 \frac{\epsilon_n}{\delta_n}\to \epsilon_0, \,\,\text{as}\,\,n\to+\infty,
\end{equation*}
for some $y_0\in\R^2$, $|y_0|\leq \tilde C$ and $\epsilon_0\geq 0$. Also, we define $D_{n}$ as in \eqref{dn2} and $\widetilde w_{n}$ as in \eqref{phicaseIA} which therefore is a solution of \eqref{equationgoodcase}.\\
At this point we argue as above and define:
$$ Q=\{(t,\phi):\;t\leq 0,\;\phi\in[0,2\pi]\},$$
and $u_n$ as in \eqref{def:un} which solves \eqref{eq:un}
where $W_n$ is defined in \eqref{def:Wn} and satisfies \eqref{monotone-V}. In particular the same argument adopted above
shows that it is well defined
$$
\mu_n=\sup\{\mu<0\,:\,\forall \nu< \mu,\; u_n(2\nu-t,\phi)-u_n(t,\phi)<0,\quad\forall\,\nu<t<0,\;\phi\in [0,2\pi]\},
$$
and that \eqref{mp-comp} holds true in this case as well.\\
Now, we want to prove that
\begin{equation}
    \label{mun:bound}
    \mu_n\leq \log\delta_n+C,
\end{equation}
for some constant $C$. Let us define
\begin{align*}
\widetilde u_n(t,\phi):&=u_n(t,\phi)+2\bns\log \delta_n \\
&=v_n(e^t\cos \phi , e^t\sin \phi)+2t-\tfrac{A}{a}e^t+2\bns\log\delta_n
\end{align*}
and
\[
u(t,\phi):=\widetilde w(e^t\cos\phi,e^t\sin \phi)+2t.
\]
Firstly, we notice that
\begin{equation}
    \label{u:boundabove}
    u(t,\phi)=\widetilde w(e^t\cos\phi,e^t\sin \phi)+2t\leq \widetilde w(y_0)+2t=2t,\quad\;\;\forall(t,\phi)\in Q.
\end{equation}
where we have used the fact that $\widetilde w(y)\leq\widetilde w(y_0)=0$, for every $y\in\R^2$.
Moreover, by using the fact that $\widetilde w_n\to\widetilde w$ uniformly on the compact sets of $\R^2$, we notice that
\begin{align*}
    \widetilde u_n(&t+\log\delta_n,\phi)-u(t,\phi)=\\
    &=v_n(\delta_ne^t\cos\phi , \delta_ne^t\sin\phi)+2(1+\bns)\log\delta_n-\tfrac{A}{a}\delta_ne^t-\widetilde w(e^t\cos\phi,e^t\sin\phi)\\
    &=\widetilde w_n (e^t\cos\phi,e^t\sin\phi)-\tfrac{A}{a}\delta_ne^t-\widetilde w(e^t\cos\phi,e^t\sin\phi)=o(1),
\end{align*}
locally uniformly as $n\to\infty$. Therefore, for every $s\in\R$ fixed,
\begin{equation}
\label{localuniformconvergence}
\sup_{\{t\leq s,\phi\in[0,2\pi]\}} |\widetilde u_n(t+\log\delta_n,\phi)-u(t,\phi)|\to 0
\end{equation}

\medskip

At this stage we study two different cases: the case $\alpha_\ii=1$ and $\widetilde\beta=16\pi$ and the case $\alpha_\ii<1$ and $\widetilde\beta<16\pi$.\\
If $\alpha_\ii=1$ and $\widetilde\beta=16\pi$ (see \eqref{equationgoodcase}), then by Theorem 2.1 in \cite{llty} (see also section \ref{sec:planar}) 
we have necessarily that $\epsilon_0=0$, which implies that $\widetilde w$ solves the following planar problem,
\begin{align*}
\graf{-\D \widetilde w= c_0|x|^2e^{\widetilde w} &\text{in}\,\,\R^2,\\
c_0\int_{\R^2}|x|^2e^{\widetilde w} =16\pi,\\
\widetilde w (y)\leq\widetilde w(y_0)=0 &\forall \,y\in\R^2. }
\end{align*}
By the result in \cite{pt}, we know the explicit expression of $\widetilde w$, which is, using complex notation,
\[
\widetilde w(y)=-2\log\left(1+\gamma_1|y^{2}-y_0^2|^2\right),\qquad y\in\mathbb C,
\]
with $\gamma_1=\tfrac{c_0}{32}$. By using this formula in $\log$-coordinates, we prove that $u$ has a strict maximum in the variable $t$ in $\R$. Indeed,
\[
u(t,\phi)=\log\left(\frac{e^{2t}}{(1+\gamma_1|e^{2t+2\phi i}-y_0^2|^2)^2}\right),\qquad (t,\phi)\in \R\times[0,2\pi)
\]
and
\[
\p_t u(t,\phi)=2\left(1-\frac{\gamma_1(4e^{4t}-4|y_0|^2\cos(2\theta-2\phi)e^{2t})}{1+\gamma_1|e^{2t+2\phi i}-y_0^2|^2}\right),
\]
with $y_0=|y_0|e^{i\theta}$, for a certain $\theta\in[0,2\pi)$. Hence,  $\p_t u(t,\phi)=0$ if and only if
\[
\widetilde t_0=\tfrac{1}{2}\log\left(\tfrac 13 |y_0|^2\cos(2\theta-2\phi)+\tfrac{1}{3}\sqrt{|y_0|^4\cos^2(2\theta-2\phi)+3(|y_0|^4+(\gamma_1)^{-1})}\right)
\]
and $u(t,\phi)\to-\infty$, as $t\to\pm\infty$, which implies that $u(\widetilde t_0,\phi)>u(t,\phi)$, for every $t\in \R$ and $\phi\in[0,2\pi)$ fixed. Here, we remark that $\tilde t_0$ depends on $\tilde C$. Now, since $\tilde t_0=\tilde t_0(\phi)$ is continuous in $\phi$, let us define 
\[
T_0:=\max_{\phi\in[0,2\pi]}\tilde t_0(\phi)
\]
and, since it is a strict point of maximum of $u(\cdot,\phi)$, we have that 
\[
\p_t u(t,\phi)<0,\qquad \forall t>T_0,\, \forall \phi\in[0,2\pi].
\]
Therefore, we have proved that, by choosing $\tilde t>T_0+1$, we have
\[
u(\tilde t-1,\phi)>u(\tilde t+1,\phi),\qquad \forall\phi\in[0,2\pi].
\]
Then, by using \eqref{localuniformconvergence}, we deduce that, for $n$ large enough,
\[
\widetilde u_n(\tilde t-1+\log\delta_n,\phi)>\widetilde u_n( \tilde t+1+\log\delta_n,\phi), 
\]
which, by definition, implies that
\[
u_n(\tilde t-1+\log\delta_n,\phi)> u_n( \tilde t+1+\log\delta_n,\phi).
\]
Hence, by choosing $\nu_n=\tilde t+\log\delta_n$ and $t_n=\tilde t+1+\log\delta_n$, we have that
\[
u_n(2\nu_n-t_n,\phi)-u_n(t_n,\phi)>0,
\]
which implies, by definition of $\mu_n$, that
\begin{equation*}
    \mu_n\leq \nu_n= \log\delta_n+\tilde t,
\end{equation*}
showing the validity of \eqref{mun:bound}.
\bigskip

Let us discuss the case: $\bis<1$ or $\widetilde\beta<16\pi$. Firstly, we want to prove that there exists $\widetilde t_0\in\R$, such that
\begin{equation}
\label{strictmaximum}
u(t,\phi)<u(\widetilde t_0,\phi) \;\quad \forall  t\neq\tilde t_0,\forall\phi\in[0,2\pi].
\end{equation}
We first prove that $u$ is a radially symmetric function. It is convenient to write $u$ in Euclidean coordinates, that is, 
for $y=(e^t\cos\phi,e^t\sin\phi)$, $(t,\phi)\in \R\times[0,2\pi)$, with an abuse of notation,
\[
u(y)=\widetilde w(y)+2\log|y|, \qquad|y|>0.
\]
It is well known (\cite{GM1}, \cite{Lin1}) that $\widetilde w$ is the unique and radially symmetric solution of \eqref{profile:tildew}. 
Therefore $u$ is radially symmetric, i.e. $u(y)=u(|y|)$.
At this point, it is easy to prove that $u$ admits at least a local maximum point in $\{|y|>0\}$.
Indeed, we notice that $\widetilde w$ is well defined and continuous in $\bar B_1$ and $u(y)\to-\infty$, when $|y|\to0$, 
which implies that for a fixed $m>0$ there exists $R_{m}$ such that $u(y)\leq-m$, for every $|y|\leq R_{m}$. 
Moreover, we rely on the result in \cite{cl2}, which states that for $|x|$ sufficiently large, we have that
\[
\widetilde w(x)=-\tfrac{\widetilde\beta}{2\pi}\log|x|+O(1),
\]
with $\widetilde\beta>8\pi$ (see section \ref{sec:planar}). Therefore, $\lim_{|y|\to+\infty}u(y)=-\infty$. 
In other words, there exists $M>0$ and $R_M>0$ such that $u(y)\leq-M$ for every $|y|\geq R_M$.
At last, $u$ is continuous in $\{R_m\leq|y|\leq R_M\}$, then  attains its maximum inside this domain. 
Hence there exists $\widetilde y_0\in\R^2$  such that $u(\tilde y_0)\geq u(y)$, for any $y\in\R^2$.\\
Now, passing to polar coordinates and letting $r_0:=|\widetilde y_0|$, we want to prove that for any $r>0$, $r\neq r_0$, $$u(r)<u(r_0).$$

By the radial symmetry of $u$, we deduce that
\[
ru'(r)=r\widetilde w'(r)+2.
\]
Therefore,
\begin{equation}
\label{relationmaxima}
u'(r)+ru''(r)=(ru'(r))'=(r\widetilde w'(r))'=-rc_0(\epsilon_0^2+r^2)^{\bis}e^{\widetilde w}<0,\qquad \forall r>0.
\end{equation}
In particular, we deduce that $u''(r_0)<0$, which implies that it is a strict local maximum point. At this point, assume by contradiction there exists another local maximum point $r_1>0$, $ r_1\neq r_0$ such that $u(r_0)=u(r_1)$. For simplicity, we choose $ r_1> r_0$. Then, by continuity of $u'$, there has to exist a point $r_2>0$, $r_0< r_2< r_1$ for which $u$ attains a local minimum, namely $u'(r_2)=0$ and $u''(r_2)\geq0$. However, by \eqref{relationmaxima}, we deduce that
\[
0\leq r_2u''(r_2)<0,
\]
which is a contradiction. Therefore, $\widetilde t_0:=\log r_0$ is a strict global point of maximum for $u(\cdot,\phi)$ in $\R$. We notice here that in this radial case $\tilde t_0$ does not depend on $\phi$.

Then, by using \eqref{localuniformconvergence} and\eqref{strictmaximum}, we deduce that, for $n$ large enough,
\begin{equation}
    \label{localuniformconvergence2}
\sup_{\{t\leq 2+\widetilde t_0,\phi\in[0,2\pi)\}}|\widetilde u_n(t+\log\delta_n,\phi)-u(t,\phi)|<C
\end{equation}
for some $C>0$ and
\[
\widetilde u_n(\widetilde t_0+\log\delta_n,\phi)> \widetilde  u_n( \widetilde t_0+2+\log\delta_n,\phi),
\]
which, by definition, implies that
\[
u_n(\widetilde t_0+\log\delta_n,\phi)> u_n( \widetilde t_0+2+\log\delta_n,\phi),
\]
Hence, by choosing $\nu_n=\widetilde t_0+1+\log\delta_n$ and $t_n=\widetilde t_0+2+\log\delta_n$, we have that
\[
u_n(2\nu_n-t_n,\phi)-u_n(t_n,\phi)>0,
\]
which implies, by definition of $\mu_n$, that
\begin{equation*}
    \mu_n\leq \nu_n=\log\delta_n+\widetilde t_0+1,
\end{equation*}
showing the validity of \eqref{mun:bound}.

Therefore,
\begin{align*}
    \inf_{B_1}\,v_n=\min_{\p B_1}\,v_n&=\min_{\phi\in[0,2\pi)}\,u_n(0,\phi)+\frac{A}{a}\\
    & \leq \max_{\phi\in[0,2\pi)}\,u_n(2\mu_n,\phi)+\frac{A}{a} \qquad &\text{(by \eqref{mp-comp})}\\
    & = \max_{\phi\in[0,2\pi)}\,\widetilde u_n(2\mu_n,\phi)-2\bns\log\delta_n+\frac{A}{a}\\
    &\leq \max_{\phi\in[0,2\pi)}\,u(2\mu_n-\log\delta_n,\phi)-2\bns\log\delta_n+O(1) \qquad &\text{(by \eqref{localuniformconvergence2})}\\
    &\leq 4\mu_n-2(1+\bns)\log\delta_n+O(1) \qquad &\text{(by \eqref{u:boundabove})}\\
    &\leq 4\log\delta_n-2(1+\bns)\log\delta_n+O(1) \qquad &\text{(by \eqref{mun:bound})}\\
    &=2(1-\bns)\log\delta_n+O(1)
\end{align*}
At last, we obtain that
\begin{align}
\label{finalHarnack3case}
   v_n(x_n)+ \inf_{B_1}\,v_n\leq -2(1+\bns)\log\delta_n+2(1-\bns)\log\delta_n+O(1),
\end{align}
which is the desired conclusion.\\

At this stage, we are still not able to prove that the constant $C_{III}$ does not depend on $\tilde C$. However, let us suppose by contradiction that there exists a family of sequence of solutions $v_n^{(k)}$ that satisfies 

\begin{equation*}
 \graf{
 -\D v_n^{(k)}=H_n^{(k)} e^{v_n^{(k)}} \quad \text{in}\quad B_1,\\
 \hfill \\
 \int_{B_1}\left({(\epsilon_n^{(k)})^2+|x|^2}\right)^{\bns^{(k)}} e^{\dis v_n^{(k)}}\leq C, }
\end{equation*}
with
\begin{equation*}\lambda_n^{(k)}=\int_{B_1}\left({(\epsilon_n^{(k)})^2+|x|^2}\right)^{\bns^{(k)}}K_n^{(k)} e^{\dis v_n^{(k)}}\to\lambda_\infty>0,
\end{equation*}
\begin{equation*}
 \bns^{(k)}:= \tfrac{\lambda_n^{(k)}}{4\pi}\sigma\leq1, \quad \bis:= \tfrac{\lambda_\infty}{4\pi}\sigma,
\end{equation*}
\begin{equation*}
H_n^{(k)}(x)=\left({(\epsilon_n^{(k)})^2+|x|^2}\right)^{\bns^{(k)}}K_n^{(k)}(x),\quad \sigma>0, \quad\epsilon_n^{(k)}\to 0,
\end{equation*}
\begin{equation*}
 K_n^{(k)}\geq 0, \,\, K_n^{(k)}\in C^{1}(\overline B_1),\quad  K_n^{(k)}\to K \quad \text{in}\,\,C^1(\overline{B}_1),
\end{equation*}
such that
$$
\max\limits_{\pa B_1} v_n^{(k)}-\min\limits_{\pa B_1} v_n^{(k)}\leq \L,
$$
and satisfies the following property,
\[
 \text{for any}\,\,r\in(0,1),\exists \, C_r>0, \,\,\text{such that}:
\]
\begin{equation*}
   \underset{\overline{B_1\backslash B_r}}{\max}\, v_n^{(k)}\leq C_r,
\end{equation*}
\begin{equation*}
    \underset{\overline{B_r}}\max\, v_n^{(k)}=v_n^{(k)}(x_n^{(k)})\to\infty,
\end{equation*}
but for which, by denoting $\delta_n^{(k)}=e^{-\frac{v_n^{(k)}(x_n^{(k)})}{2(1+\bns^{(k)})}}\to0^+$ and
\[
D_{n}^{(k)}:=v_n^{(k)}(x_n^{(k)})+\inf_{B_1}v_n^{(k)}+4\bns^{(k)}\log\delta_n^{(k)},
\]
we have that
\begin{align}
\label{contradictionDnk}
\limsup_{n\to+\infty}D_{n}^{(k)}\to+\infty,\qquad\text{as}\,\,k\to+\infty.
\end{align}

Now, passing eventually to a subsequence, we can choose $n_k$ sufficiently large such that
\[
D_{n_k}^{(k)}\geq k.
\]
Moreover, we can choose $n_k$ such that 
\[
\epsilon_{n_k}^{(k)}\leq\tfrac{1}{k},\qquad|x_{n_k}^{(k)}|\leq\tfrac{1}{k},\qquad \delta_{n_k}^{(k)}\leq\tfrac{1}{k},
\]
\[
\|K_{n_k}^{(k)}-K\|_{C^1}\leq\tfrac{1}{k},\qquad |\lm_{n_k}^{(k)}-\lm_\ii|\leq\tfrac1k
\]
and such that we preserve the previously mentioned concentration properties away from the origin. \\
Now, passing to a diagonal subsequence $v_k:=v_{n_k}^{(k)}$, this sequence satisfies the hypotheses of Theorem \ref{supinfI} and
\[
D_k\to+\infty,\qquad\text{as}\,\,k\to+\infty.
\]
Let $$Q_{n}^{(k)}:=\max\left\{\frac{\epsilon_n^{(k)}}{t_n^{(k)}},\frac{|x_n^{(k)}|}{\delta_n^{(k)}}\right\},$$
where $t_n^{(k)}:=\max\{|x_n^{(k)}|,\delta_n^{(k)}\}$. 
At this stage, if, along a subsequence,
\[
Q_k\leq R,
\]
with $R>0$ fixed, then we are in the hypotheses of CASE (III) and by \eqref{finalHarnack3case} we have
\[
D_k\leq C,
\]
with $C$ uniform in $k$, which implies a contradiction.
Otherwise, if $Q_k\to+\infty,$ we have two possibilities: if, along a subsequence,
\[
\tfrac{\epsilon_k}{t_k}\to+\infty, \qquad\text{as}\,\,k\to+\infty,
\]
then $v_k$ satisfies the hypothesis of the CASE (I) of Theorem \ref{supinfI}; then,
\[
D_k\leq v_k(x_k)+\inf_{B_1}v_k+4\alpha_{k,\sg}\log\epsilon_k+4\alpha_{k,\sg}\log(\tfrac{\delta_k}{\epsilon_k})\leq C_I,
\]
for $n$ large enough, which is a contradiction. On the other hand, if
\[
\tfrac{\epsilon_k}{t_k}\leq C, 
\]
uniformly in $k$, then $v_k$ satisfies the hypotheses of the CASE (II) of Theorem \ref{supinfI}; then,
\[
D_k\leq v_k(x_k)+\inf_{B_1}v_k+4\alpha_{k,\sg}\log|x_k|+4\alpha_{k,\sg}\log\tfrac{\delta_k}{|x_k|}\leq C_{II},
\]
for $n$ large enough, which is a contradiction. This proves that the constant $C_{III}$ does not depend on $\tilde C$ and completes the proof of Theorem \ref{supinfI}.
\finedim

\bigskip
\bigskip
At this point, in view of Theorem \ref{supinfI}, we can prove Corollary \ref{cor:supinf}.\\
{\it The Proof of Corollary \ref{cor:supinf}.}\\
If, possibly along a subsequence, $\bns=1$, then the inequality easily follows since \eqref{lambdan} holds and consequently 
$\inf_{B_1} v_n\leq C$. Hence, we can assume without loss of generality that for any $n$,
$\bns<1$. If $v_n$ were uniformly bounded in a neighborhood of $x=0$, the conclusion would immediately follow from the ``Sup+Inf" inequality in \cite{bls}.
Thus, there is no loss of generality, after a suitable rescaling, in assuming that \eqref{blowup:hyp},
\eqref{bounded} and \eqref{explosion} hold true. At this point, by Theorem \ref{Concentration}, we have $\sg\leq \frac12$ and in particular
all the assumption of Theorem \ref{massboundarycontrol} are satisfied. Actually, from \eqref{blowup:hyp} and for any $n$ large enough we have
that $\sup_{B_\frac12} v_n=v_n(x_n)$ and we can apply Theorem \ref{supinfI}.\\
Assume first that \eqref{sup+inf-I} holds true, then we have,
$$
\inf\limits_{B_1}v_n\leq C_I+\log\left(\frac{\dt^{2(1+\bns)}_n}{\eps^{4\bns}_n}\right),
$$
and consequently
\begin{align*}
    \tfrac{1+\bns}{1-\bns}\inf\limits_{B_1}v_n\leq & \tfrac{1+\bns}{1-\bns}C_I+\tfrac{1+\bns}{1-\bns}\log(\tfrac{\dt^{2(1+\bns)}_n}{\eps^{4\bns}_n})\\
    =&  -v_n(x_n)+\tfrac{1+\bns}{1-\bns}C_I+\tfrac{2\bns}{1-\bns}\log({\dt^{2(1+\bns)}_n})+\tfrac{1+\bns}{1-\bns}\log(\tfrac{1}{\eps^{4\bns}_n}) \\
    =& -v_n(x_n)+\tfrac{1+\bns}{1-\bns}C_I+\tfrac{1+\bns}{1-\bns}\log({\dt^{4\bns}_n})+\tfrac{1+\bns}{1-\bns}\log(\tfrac{1}{\eps^{4\bns}_n}) \\
    = & -v_n(x_n)+\tfrac{1+\bns}{1-\bns}C_I+4\bns\tfrac{1+\bns}{1-\bns}\log(\tfrac{\dt_n}{\eps_n}),
\end{align*}
and the conclusion follows since $\frac{\dt_n}{\eps_n}\to 0^+$.\\

If \eqref{sup+inf-II} holds true, we can proceed similarly as in Case I, with $\epsilon_n$ replaced by $|x_n|$. This yields the desired estimate,
\[
\tfrac{1+\bns}{1-\bns}\inf\limits_{B_1}v_n\leq -v_n(x_n)+\tfrac{1+\bns}{1-\bns}C_{II}+4\bns\tfrac{1+\bns}{1-\bns}\log(\tfrac{\dt_n}{|x_n|}),
\]
and the conclusion follows by noting that $\frac{\dt_n}{|x_n|}\to 0^+$.

If \eqref{sup+inf-III} holds true we have,

$$
\inf\limits_{B_1}v_n\leq C_{III}+2(1-\bns)\log\dt_n,
$$
and consequently
\[
\tfrac{1+\bns}{1-\bns}\inf\limits_{B_1}v_n\leq \tfrac{1+\bns}{1-\bns}C_{III}+2(1+\bns)\log\dt_n= \tfrac{1+\bns}{1-\bns}C_{III}-v_n(x_n),
\]
and the conclusion immediately follows, in this case as well.

\finedim

\bigskip
\bigskip

\bigskip

\section{``Slow" blow up: Harnack-type inequalities and a quantization Lemma for CASE (III).}\label{sectslowblowup}
As mentioned in the introduction, due to the fact that $\bis=1$, the decay obtained via the Harnack-type estimates in Theorem \ref{supinfI} is
not enough to come up with the mass-quantization in CASE (III). A solution to this issue comes from a careful adaptation of an argument in \cite{CLin0} (see also \cite{cosentino2025}).
\begin{lemma}\label{lemmasup+Cinfspecialcase}
Let $u_n$ be a sequence of solutions of
\begin{equation*}\label{equationsup+Cinfalpha=1}
\begin{cases}
    -\D u_n=\left({\varepsilon_n^2+|x|^2}\right)^{\alpha_n}V_n(x)e^{\displaystyle u_n} \qquad\text{in} \,\,\, B_1,\\
    \int_{B_1}\left({\varepsilon_n^2+|x|^2}\right)^{\alpha_n}e^{\displaystyle u_n}\leq C
\end{cases}
\end{equation*}
with $\varepsilon_n\to0^+$, $\alpha_n\leq1$,
$\alpha_n\to \alpha_\infty\in(0,1]$ and $V_n$ satisfying
\begin{equation*}\label{Vn0sup+Cinfalpha=1}
 V_n\geq 0, \quad V_n\in C^{0}(\overline B_1), \quad V_n\to V \,\, \text{uniformly in}\,\,\overline{B_1}.
\end{equation*}
Let $K\subset B_1$ be any compact set such that $0\in K$ and let $x_n$ be a maximum point of $u_n$ in $K$. Assume that
\begin{equation*}
\label{blowupin0}
x_n\to0 \qquad\text{and}\qquad\sup_{B_1}\,u_n=\sup_{K}\,u_n=u_n(x_n)\to+\infty.
\end{equation*}
Let us define
\[
\delta_n=e^{-\frac{u_n(x_n)}{2(1+\alpha_n)}}\to0^+
\]
and assume that,
\[
\frac{\varepsilon_n}{\delta_n}\to\epsilon_0\qquad \text{and}\qquad \frac{x_n}{\delta_n}\to y_0,
\]
for some $\epsilon_0\geq0$ and $y_0\in\R^2$. Let $0<d_0\leq\min\left\{\tfrac{1}{8},\tfrac{1}{4}\left(\tfrac{\widetilde\beta}{8\pi}-1\right)\right\}$, where $\widetilde \beta$ is defined in Remark \ref{remarksup+Cinf}.
Then,
\begin{equation}
 \label{CaseIIImultiplebubbling:sup+Cinf}
 \underset{K}\sup\, u_n + C_1\,\underset{B_1}\inf\, u_n\leq 0,
\end{equation}
for every $C_1\geq\left(1+\tfrac{2}{d_0}\right)$ and for every $n$ large enough.
\end{lemma}

\begin{remark}
    \label{remarksup+Cinf}
Under the assumption of Lemma \ref{lemmasup+Cinfspecialcase}, let us define
\[
\widetilde u_n(z)=u_n(\delta_nz)+2(1+\alpha_n)\log\delta_n,\qquad z\in B_{\tfrac{1}{\delta_n}}(0).
\]
Then, $\widetilde u_n$ satisfies
\begin{align}\label{eqtildeun}
\begin{cases}
-\D \widetilde u_n= (\tfrac{\varepsilon_n^2}{\delta_n^2}+|z|^2)^{\alpha_n}V_n(\delta_nz)e^{\dis \widetilde u_n}
\quad\text{in}\quad B_{\tfrac{1}{\delta_n}}(0),\\
\int_{B_{\frac{1}{\delta_n}}(0)}(\tfrac{\varepsilon_n^2}{\delta_n^2}+|z|^2)^{\alpha_n}e^{\dis \widetilde u_n}\leq C\\
\widetilde u_n(z)\leq\widetilde u_n(\tfrac{x_n}{\delta_n})=\underset{B_{\frac{1}{\delta_n}}(0)}{\max}\widetilde u_n=0,\qquad \forall z\in B_{\tfrac{1}{\delta_n}}(0).
\end{cases}
\end{align}

By the regular Concentration-Compactness alternative \cite{bm} (in the case $\epsilon_0\neq0$) or by Theorem \ref{Concentration-Compactness alternative} (in the case $\epsilon_0=0$) and by standard elliptic estimates, we easily conclude that $\widetilde u_n$ is uniformly
bounded in $C^{t}_{\rm loc}(\R^2)$, for some $t\in(0,1)$ and then, possibly along a subsequence, $\widetilde u_n\to \widetilde w$
uniformly on compact sets of $\R^2$, where $\widetilde w$ satisfies
\begin{align*}
\begin{cases}
 -\D \widetilde w=(\epsilon_0^2+|z|^2)^{\alpha_\infty} V(0)e^{\dis \widetilde w} \quad\text{in}\quad\R^2 \\
\int_{\R^2}(\epsilon_0^2+|z|^2)^{\alpha_\infty}e^{\dis \widetilde w}\leq C, \\
\widetilde w(z)\leq\widetilde w(y_0)=0, \qquad\forall z\in\R^2.
\end{cases}
\end{align*}
Since $\widetilde w$ is bounded from above, then necessarily $V(0)\neq0$. At this point, see Lemma \ref{massstrange}, we have that,
\begin{align*}
\widetilde\beta:=\underset{\R^2}\int(\epsilon_0^2+|z|^2)^{\alpha_\ii}V(0)e^{\dis \widetilde w}> \max\{8\pi,4\pi(1+\alpha_\infty)\}= 8\pi.
\end{align*}

\end{remark}

{\it The Proof of Lemma \ref{lemmasup+Cinfspecialcase}.}\\
We argue by contradiction and, by fixing $C_1\geq\left(1+\tfrac{2}{d_0}\right)$, we assume that
\[
u_n(x_n)+C_1\inf_{B_1}u_n>0.
\]
Let us set $\rho>0$ such that $B_\rho(x_n)\subset B_1(0)$. By the Green representation formula for $u_n$ on $B_\rho(x_n)$, we have that,
\begin{align}
\nonumber
u_n(x_n)&=\int\limits_{B_{\rho}(x_n)}G_{B_{\rho}(x_n)}(x_n,y)\left({\varepsilon_n^2+|y|^2}\right)^{\alpha_n}V_n(y)e^{\displaystyle u_n}\,dy+\underbrace{\int\limits_{|y-x_n|=\rho}-\frac{\p G_{B_{\rho}(x_n)}}{\p r}(x_n,y)\,u_n(y)\,d\sigma(y)}_{p_n}\\
&=\int_{B_{\frac{\rho}{\delta_n}}(\frac{x_n}{\delta_n})}G_{B_{{\rho}}(x_n)}(x_n,\delta_nz)\left({\tfrac{\varepsilon_n^2}{\delta_n^2}+|z|^2}\right)^{\alpha_n}V_n(\delta_nz)e^{\displaystyle \widetilde u_n}\,dz+p_n \label{Greenrepresentationun}
\end{align}
where $\widetilde u_n$ is defined in Remark \ref{remarksup+Cinf} and,
putting
 $\bar x=x_n+\tfrac{\rho^2}{|x-x_n|^2}(x-x_n)$,
\begin{align*}
    G_{B_{\rho}(x_n)}(x,y)=\begin{cases}
        -\tfrac{1}{2\pi}\log|x-y|+\tfrac{1}{2\pi}\log\tfrac{|\bar x-y||x-x_n|}{\rho},\,\,&x\neq x_n, \\
        -\tfrac{1}{2\pi}\log\tfrac{|x_n-y|}{\rho},\,\,&x=x_n,
    \end{cases}
\end{align*}
\begin{align*}
    \frac{\p G_{B_{\rho}(x_n)}}{\p r}(x,y)=\frac{1}{2\pi\rho}\frac{|x-x_n|^2-\rho^2}{|x-y|^2}.
\end{align*}
Now, for a fixed $0<d_0\leq\min\left\{\tfrac{1}{8},\tfrac{1}{4}\left(\tfrac{\widetilde\beta}{8\pi}-1\right)\right\}$, there exist $L>0$ and $N_1\in\N$ large enough, such that for every $l\geq L$ and $n\geq N_1$ we have that
\[
\underset{|z-\frac{x_n}{\delta_n}|\leq l}\int(\tfrac{\varepsilon_n^2}{\delta_n^2}+|z|^2)^{\alpha_n} V_n(\delta_n z)e^{\dis \widetilde u_n}\geq \widetilde\beta(1-d_0).
\]
Moreover, there exists $N_2\in \N$ such that, for every $n\geq N_2$,
\[
\rho(\tfrac{1}{\delta_n})^{\frac{d_0}{1-d_0}}\in(L,\tfrac{\rho}{\delta_n}).
\]
Then, for every $l\in(L,\rho(\tfrac{1}{\delta_n})^{\frac{d_0}{1-d_0}})$, we deduce that
\[
|z-\tfrac{x_n}{\delta_n}|\leq l \implies -\frac{1}{2\pi}\log\left(\frac{|\tfrac{x_n}{\delta_n}-z|}{\rho}\right)\geq-\frac{d_0}{1-d_0}\frac{u_n(x_n)}{4\pi(1+\alpha_n)}.
\]
Then, by \eqref{Greenrepresentationun} and taking $l\in(L,\rho(\tfrac{1}{\delta_n})^{\frac{d_0}{1-d_0}})$ and $n\ge\max\{N_1,N_2\}$, we have that
\begin{align*}
u_n(x_n)&=-\tfrac{1}{2\pi}\int_{B_{\frac{\rho}{\delta_n}}(\frac{x_n}{\delta_n})}\log\left(\frac{|x_n-\delta_n z|}{\rho}\right)(\tfrac{\varepsilon_n^2}{\delta_n^2}+|z|^2)^{\alpha_n} V_n(\delta_n z)e^{\dis \widetilde u_n}\,dz+p_n\\
&= \int_{B_{\frac{\rho}{\delta_n}}(\frac{x_n}{\delta_n})}\left(\frac{u_n(x_n)}{4\pi(1+\alpha_n)}-\frac{1}{2\pi}\log\left(\frac{|\tfrac{x_n}{\delta_n}-z|}{\rho}\right)\right)(\tfrac{\varepsilon_n^2}{\delta_n^2}+|z|^2)^{\alpha_n} V_n(\delta_n z)e^{\dis \widetilde u_n}\,dz+p_n\\
&\geq \left(1-\frac{d_0}{1-d_0}\right)\frac{u_n(x_n)}{4\pi(1+\alpha_n)}\int_{B_{l}(\frac{x_n}{\delta_n})}(\tfrac{\varepsilon_n^2}{\delta_n^2}+|z|^2)^{\alpha_n} V_n(\delta_n z)e^{\dis \widetilde u_n}+p_n\\
&\geq \left(1-2d_0\right)\frac{\widetilde\beta}{4\pi(1+\alpha_n)}u_n(x_n)+p_n\\
&\geq \left(1-2d_0\right)\frac{\widetilde\beta}{8\pi}u_n(x_n)+p_n\\
&\geq \left(1-2d_0\right)(1+4d_0)u_n(x_n)+p_n\\
&\geq (1+\tfrac{d_0}{2+d_0})u_n(x_n)+p_n\\
&\geq (1+\tfrac{1}{C_1})u_n(x_n)+p_n.
\end{align*}
This, together with the fact that $p_n\geq\inf_{B_1}u_n$, implies
\[
\tfrac{1}{C_1}\, u_n(x_n)+\inf_{B_1}u_n\leq0,
\]
that is
\[
u_n(x_n)+C_1\inf_{B_1}u_n\leq0,
\]
which is our contradiction.
\finedim

\bigskip

In view of Lemma \ref{lemmasup+Cinfspecialcase}, we can prove the following mass-quantization result as in \cite{ls}. The proof goes along the same argument
in \cite{ls}, whence, also for ease of the presentation, we postpone it to subsection \ref{quant:III} below.
\begin{lemma}[A quantization Lemma for CASE (III)] \label{LS-quant} Let $v_n$ be a sequence of solutions of \eqref{eqbase}, \eqref{lambdan}, where
$K_n$ satisfies \eqref{K+} and
$$\lm_n\leq\tfrac{4\pi}{\sigma},\qquad\lambda_\ii=\tfrac{4\pi}{\sg}\qquad \text{and}\qquad\sigma\in(\tfrac14,\tfrac12).$$
Assume that \eqref{bounded} and \eqref{explosion}
are satisfied and that  there exists $d_n\to 0^+$ with the following properties:\\
$$
\frac{\eps_n}{\delta_n}\leq C,\qquad \frac{|x_n|}{\dt_n}\leq C,\qquad \frac{\dt_n}{d_n}\to 0,
$$
$$
\int\limits_{B_{8d_n}}H_n e^{\dis v_{n}}\to \beta> 8\pi,
$$
and
\begin{equation}
\label{hypols-lem}
v_n(x)+2(1+\al_{n,\sg})\log(|x|)\leq C, \;\forall\, 4d_n\leq|x|\leq 1,
\end{equation}
for some $C>0$. Then
$$
\int\limits_{B_{1}\setminus B_{8d_n}}H_n e^{\dis v_{n}}\to 0\quad\mbox{and}\quad
\lm_n\to \lm_\ii=\beta.
$$
\end{lemma}

\bigskip
\bigskip

\section{The proof of Theorem \ref{thm:13}}
In this section we present the proof of Theorem \ref{thm:13}.\\
{\it The proof of Theorem \ref{thm:13}.} Clearly all the assumptions of Theorems \ref{Concentration} and \ref{Concentration-Compactness alternative}
are satisfied, whence in particular $x=0$ is the unique blow up point in $B_1$, $0<\sg\leq \frac12$ and $\lm_\ii\geq 8\pi$.
In particular Theorem \ref{massboundarycontrol} can be applied and we also have $\lm_\ii\leq\min\left\{\frac{8\pi}{1-2\sg},\frac{4\pi}{\sg}\right\}$.
Therefore, since by assumption $\lm_\ii=\frac{4\pi}{\sg}$, we are just left with the case $\sg\in \left[\frac14,\frac12\right]$.\\
Assume first that \eqref{Hypothesis:CaseI} holds true, then by Theorem \ref{quant:imp} we have that
$$\sg=\frac12\,\,\, \text{and}\,\,\,\lm_\ii=8\pi.$$
Next, let us assume that \eqref{Hypothesis:CaseII} holds true, then from Theorem \ref{quant:imp} and Proposition \ref{caseII.quant} we see that
$$\sg=\frac14\,\,\,\text{and}\,\,\, \lm_\ii=16\pi.$$\\
At last we assume that \eqref{Hypothesis:CaseIII} holds true, whence along a subsequence we have,
\begin{equation*}
\frac{\eps_n}{\dt_n} \to \eps_0\geq 0, \quad \frac{x_n}{\dt_n} \to y_0\in \R^2.
\end{equation*}
Let $\widetilde{w}_n$ be defined as in \eqref{phicaseIA} which satisfies \eqref{equationgoodcase}.
Clearly the concentration-compactness alternative of regular Liouville type equations (\cite{bm}) can be applied to $\widetilde w_n$
and we readily deduce that $\widetilde w_n$ is uniformly bounded in  $L^{\infty}_{\rm loc}(\R^2)$.
Therefore, by standard elliptic estimates, possibly along a subsequence we have that $\widetilde w_n\to \widetilde w$ in $C^{2}_{\rm loc}(\R^2)$,
where $\widetilde w$ is a solution of \eqref{profile:tildew}. Since by assumption $\bis= 1$ and consequently
$\max\{8\pi,4\pi(1+\bis)\}=8\pi$, by Lemma \ref{massstrange} below, we have that
\begin{align*}
8\pi<\widetilde \beta \leq 8\pi(1+\bis)=16\pi.
\end{align*}
However, (see also section \ref{sec:planar}) since $\bis= 1$, it is well known (\cite{GM1}, \cite{Lin1}) that either $\eps_0>0$ and then $8\pi<\widetilde \beta <16\pi$ and
$\widetilde w$ is unique, radially symmetric and nondegenerate, or $\eps_0=0$ and then $\widetilde \beta =16\pi$ and
$\widetilde w$ can be either radially symmetric or nonradial with two maximum points (see \cite{pt}). Therefore there are only two possibilities:\\
either $\sg=\frac14$ and $\lm_\ii=16\pi$ or $\sg\in (\frac14, \frac12)$ and $\lm_\ii=\frac{4\pi}{\sg}\in (8\pi,16\pi)$. Assume we are in the latter case.
Therefore we readily deduce that there exists $\mu>0$ and a sequence $R_n\to +\ii$ such that $R_n\delta_n\to0^+$,
\[
\|\tilde w_n - \tilde w\|_{C^{2}(B_{4R_n})}\to 0,
\]
\begin{equation}\label{masstilde0}
\lim\limits_{n\to +\ii}\int_{B_{4R_n}}\widetilde V_n e^{\dis \widetilde w_n} =\lim\limits_{n\to +\ii}\int_{B_{2R_n}}
\widetilde V_n e^{\dis \widetilde w_n}= \tilde \beta\geq 8\pi+2\mu
\end{equation}
and
\begin{equation*}
\int_{D_n\setminus B_{4 R_n}}\widetilde V_n e^{\dis \widetilde w_n}\leq
\int_{D_n\setminus B_{2 R_n}}\widetilde V_n e^{\dis \widetilde w_n} \leq 8\pi-\mu.
\end{equation*}
We show that there exist $C_{r}>0$ such that
\begin{equation}
\label{hypolsII-B}
 \sup_{B_{2r}\backslash B_{2R_n\dt_n}}\{v_n(z)+2(1+\bns)\log |z|\}\leq C_{r}.
\end{equation}
By contradiction, assume that there exists a sequence $z_n$ such that
\[
 v_n(z_n)+2(1+\bns)\log |z_n|\to+\infty, \,\,\,\,\,\text{as}\,\,n\to+\infty,
\]
\[
 \mbox{and }\quad 2 R_n \dt_n\leq |z_n|\leq 2r.
\]

Then, let us define
\[
 \tilde v_n(z)= v_n(|z_n|z)+2(1+\bns)\log |z_n|,
\]
which satisfies
\begin{align*}
\begin{cases}
-\D  \tilde v_{n}= \tilde V_{n,1}(z)e^{\dis \tilde v_{n}} \qquad\text{in}\,\,B_{\frac{2r}{|z_{n}|}}(0),\\
 \tilde V_{n,1}(z)=\left({\frac{\epsilon_n^2}{|z_{n}|^2}+|z|^2}\right)^{\bns}K_n(|z_n|z)
\to c_0|z|^{2\bis}\;\mbox{in}\;C^t_{loc}(\R^2),\\
\int_{B_{\frac{2r}{|z_n|}}} \tilde V_{n,1} e^{\dis  \tilde v_{n}}= \lm_n +o(1) = \lm_\ii +o(1),\\
 \tilde v_{n}(\tfrac{z_n}{|z_n|})=v_n(z_n)+2(1+\bns)\log |z_n|\to+\infty,
\end{cases}
\end{align*}
where we used $c_0=K(0)$ and
\[
 \tfrac{\epsilon_n}{|z_n|}=\tfrac{\epsilon_n}{\dt_n}\tfrac{\dt_n}{|z_n|}\leq\tfrac{C}{2 R_n}\to 0.
\]
Possibly along a subsequence we can assume that $\tfrac{z_n}{|z_n|}\to \tilde z,\;|\tilde z|=1$ and then, as in Remark \ref{rem:reg},
the concentration-compactness theory (\cite{bm},\cite{yy}) for regular Liouville type equations applies to $\tilde v_{n}$ in $B_R(\tilde z)$ for any
$R\leq \frac12$ and in particular we have that,
\begin{equation*}
  \underset{\p B_R(\tilde z)}\max\tilde v_n-\underset{\p B_R(\tilde z)}\min\tilde v_n\leq C_R,\quad \forall \, R\leq\tfrac{1}{2},
\end{equation*}
whence $\tilde z$ is a simple blow up point (\cite{yy}), i.e.
$$
\int\limits_{B_R(\tilde z)} \tilde V_{n,1} e^{\tilde v_n}\to 8\pi,
$$
for any $R\leq \frac12$. Remark that for any $\rho\leq \frac14$ we have,
$$
\mbox{dist}(B_{R_n\dt_n}(x_n),B_{\rho|z_n|}(z_n))\geq R_n\dt_n
\left(\frac{|z_n|}{R_n |\dt_n|}(1-\rho)-(1+\frac{|x_n|}{R_n\dt_n})\right)\geq
$$
$$
R_n\dt_n(2\frac{3}{4}- 1+o(1))>0
$$
and then we have,
$$
16\pi>\lm_\ii\geq \liminf\limits_{n\to +\ii} \int\limits_{B_{R_n}(\frac{x_n}{\delta_n})} \tilde V_{n} e^{\dis \tilde w_n}+
\liminf\limits_{n\to +\ii} \int\limits_{B_r(\tilde z)} \tilde V_{n,1} e^{\dis \tilde v_{n}}\geq16\pi +2\mu,
$$
which is a contradiction and proves \eqref{hypolsII-B}. Therefore, in view of \eqref{masstilde0} and  \eqref{hypolsII-B}, we can apply Lemma \ref{LS-quant} with
$$
d_n=\frac{1}{2} R_n\dt_n,
$$
 to deduce that,
 \begin{equation}\label{mass-IIb1}
 \lm_n=\widetilde \beta+o(1) \quad \mbox{and}\quad \lm_\ii= \widetilde \beta\in (8\pi,16\pi).
 \end{equation}
We claim that
\begin{equation}\label{profilewtilde:H11}
\widetilde w_n(y)=\widetilde U_n(y;\widetilde M_n) + O(1),\quad y\in D_n,
\end{equation}
where, for any large $R\geq 1$ we have,
\begin{equation}\label{profilevtilde:H11}
\widetilde U_n(y;\widetilde M_n)=\graf{\widetilde w(y)+O(1),\quad |y|\leq R, \\ -\frac{\widetilde M_n}{2\pi}\log(|y|) + O(1),\quad R\leq  |y|\leq r\delta_n^{-1},}
\end{equation}
\[
 \widetilde M_n=\int_{D_{n}}\widetilde V_n e^{\dis \widetilde w_n}=\lm_n+o(1),
\]
and $\widetilde w$ is the unique solution of \eqref{profile:tildew}.
The argument is a slightly simpler version of the argument used in the Appendix of \cite{BCW1} about CASE (I) when $\lm_\ii<\tfrac{4\pi}{\sg}$ and much in the spirit of similar well known estimates obtained in
\cite{bclt} and \cite{tar-sd}, where one just replaces the radial profiles in
\cite{pt} with the profile \eqref{profilevtilde:H11}. Therefore, we omit the details of this proof here to avoid repetitions.\\
At this point we wish to show that \eqref{profilevtilde:H11} still holds true with  $\widetilde M_n$ replaced by $\lm_n$.
First of all, in terms of $v_n$, \eqref{profilewtilde:H11} takes the form,
\begin{equation}\label{profilevtilde:H1}
 v_n(x)=v_n(x_n)+ \widetilde U_n(\delta_n^{-1}x;\widetilde M_n)+O(1), \quad x\in B_{r}(0),
\end{equation}
and implies, in view of \eqref{mass-IIb1} and \eqref{profilevtilde:H11}, that
$$
v_n(x)=\left(1-\frac{\widetilde M_n}{4\pi(1+\bns)}\right)v_n(x_n)+O(1),\quad 2 |x|=r.
$$
On the other side, since putting $m_n=\inf\limits_{B_1} v_n$ and in view of \eqref{oscillation} we have that,
\begin{equation*}
v_n(x)-m_n=\lm_\ii G_{B_1}(x,0)+O(1),\quad \frac14\leq |x|\leq 1,
\end{equation*}
then we deduce that,
\beq\label{vn:cn-IIb1}
v_n(x)=O(1)+m_n,\quad 2 |x|=r,
\eeq
implying in turn by \eqref{masstilde0} that,
\beq\label{eps:II-B1}
m_n=O(1)-\left(\frac{\widetilde \beta+o(1)}{4\pi(1+\bns)}-1\right)v_n(x_n)\to -\infty.
\eeq
Thus, again because of \eqref{masstilde0}, \eqref{vn:cn-IIb1} and $\widetilde\beta>8\pi$, we see that
\[
0\leq \lm_n-\widetilde M_n= \int\limits_{B_1\setminus B_r}H_n e^{\dis v_n(x)}\leq
C_r\dt_n^{ \frac{\tilde \beta}{2\pi}  -2(1+\bns)}\leq C_r\dt_n^{ \frac\mu4},
\]
implying that, for any $1\leq R\leq |y| \leq r\dt_n^{-1}$, we have that
$$
|\lm_n-\widetilde M_n||\log(|y|)|\leq C\dt_n^{ \frac\mu4}\log(\dt_n^{-1})=o(1).
$$
As a consequence, it is readily seen that \eqref{profilevtilde:H11} still holds true with $\widetilde M_n$ replaced by $\lm_n$,
which in particular in terms of $v_n$ takes the form
\eqref{profilevtilde:H1-IIb}.
\finedim

\section{Appendix}

\subsection{Known results about \eqref{profile:tildew}}\label{sec:planar}

Here we list some facts about the problem \eqref{profile:tildew}. For more details we refer to the Appendix in \cite{BCW1}.
In particular, we will need the following well known facts  (\cite{cl2}, \cite{det}).
\begin{lemma}\label{massstrange}
 Let $\psi$ be a solution of the following problem
 \begin{equation*}
 \begin{cases}
 -\D \psi(x)= (\varepsilon_0^2+|x|^2)^{\alpha}e^{\dis \psi}\quad\text{in}\,\, \R^2, \\
 \underset{\R^2}\int(\varepsilon_0^2+|x|^2)^{\alpha}e^{\dis \psi}<+\infty,
 \end{cases}
\end{equation*}
for $\varepsilon_0\geq0$ and $\alpha>-1$. Let us denote by
\begin{equation*}
 \beta=\tfrac{1}{2\pi}\underset{\R^2}\int(\varepsilon_0^2+|x|^2)^{\alpha}e^{\dis \psi}.
\end{equation*}
Then,
\begin{itemize}
 \item If $\alpha>0$, $(\varepsilon_0^2+|x|^2)^{\alpha}e^{\dis \psi}$ is bounded in $\R^2$,
 \item There exists a constant $C$ such that
 \begin{equation*}\label{asymptotic}
  -\beta\ln(|x|+1)-C\leq \psi(x)\leq-\beta\ln(|x|+1)+C,
 \end{equation*}
 \item The following identity holds,
 \begin{equation*}\label{relation}
  2\alpha \underset{\R^2}\int|x|^2(\varepsilon_0^2+|x|^2)^{\alpha-1}e^{\dis \psi}=\pi\beta(\beta-4).
 \end{equation*}
 \item If $-1<\alpha< 0$, then
 \begin{equation*}
 \label{massnegative}
  8\pi(1+\alpha)\leq \underset{\R^2}\int(\varepsilon_0^2+|x|^2)^{\alpha}e^{\dis \psi}<8\pi,
 \end{equation*}
 where the l.h.s equality holds if and only if $\varepsilon_0=0$.\\
 If $\alpha\geq 0$, then
 \begin{equation*}
 \label{masspositive}
\max\{8\pi,4\pi(1+\alpha)\}\leq \underset{\R^2}\int(\varepsilon_0^2+|x|^2)^{\alpha}e^{\dis \psi}\leq  8\pi(1+\alpha),
 \end{equation*}
 where the l.h.s. equality holds if and only if $\alpha=0$ while the right hand side equality holds if and only if $\varepsilon_0=0$ or $\alpha=0$.
\end{itemize}
\end{lemma}

\bigskip
\bigskip

\subsection{The proof that $0$ cannot be a blow up point in CASE II}

\begin{proof}
We recall from the Pohozaev argument used in Theorem \ref{quant:imp} that 
\begin{equation}
\label{Pohozaevappendix}
 J=-\tfrac{\lm_\ii^2}{4\pi}+2(1+\bis)\lm_\ii
\end{equation}
where we have defined
\[
J_{n,r}:=2\bns\int_{B_r}\tfrac{\epsilon_n^2}{\epsilon_n^2+|x|^2} H_ne^{\dis v_n}\,dx,\qquad\text{and}\qquad J:=\lim_{r\to0^+}\lim_{n\to+\infty} J_{n,r}.
\]
Let us suppose that $0$ is a blow up point for $\bar w_n$. Then, by Theorem \ref{Concentration-Compactness alternative} in a small neighborhood of the origin we have
\begin{equation}
\label{convergence0}
\bar V_ne^{\dis \bar w_n}\rightharpoonup\beta_0\delta_{p=0}\qquad\text{in}\,\,B_{\eta}(0),
\end{equation}
for $0<\eta<1$.
 We already know that $\bar z_0$, with $|\bar z_0|=1$, is a blow up point for $\bar w_n$. By the concentration-compactness alternative (see Remark \ref{rem:concentrationgeneral}), every blow up point carries at least $8\pi$ of mass and, since \eqref{boundsbeta} holds and
 \[
 \lm_\ii+o(1)=\int_{D^{(n)}}\bar V_ne^{\dis \bar w_n}\geq \int_{B_{\frac{\eta}{2}}(0)}\bar V_ne^{\dis \bar w_n}+\int_{B_{\frac{\eta}{2}}(\bar z_0)}\bar V_ne^{\dis \bar w_n}=\beta_0+8\pi+o(1),
 \]
we have
\[
16\pi\leq 8\pi+\b_0\leq \lm_\ii\leq 8\pi(1+\bis)\leq16\pi.
\]
Therefore, we have that
\[
\beta_0=8\pi,\qquad\lm_\ii=16\pi,\qquad\bis=1.
\]
As a consequence, from \eqref{Pohozaevappendix}, we deduce that
\[
J=0
\]
and, by Theorem \ref{quant:imp}, necessarily $\frac{\epsilon_n}{|x_n|}\to\bar \epsilon_0=0$.
Now, we want to apply the Pohozaev argument for $\bar w_n$ in $B_\eta(0)$.

Let us define $\bar s_n$ to be the unique solution of,
\begin{equation*}
 \begin{cases}
  -\D \bar s_n=0\quad &\text{in}\,\,B_\eta,\\
  \bar s_n=\bar w_n-\underset{\p B_\eta}\min\,\bar w_n\quad &\text{on}\,\,\p B_\eta.
 \end{cases}
\end{equation*}
By standard elliptic estimates, since \eqref{osc:barwnU} holds, we have that
\begin{equation*}
 \|\bar s_n\|_\infty\leq C,
\end{equation*}
for a suitable positive constant $C$, and, along a subsequence, we may assume that $\bar s_n\to \bar s$ in $C^1_{loc}(B_\eta)$. Now, let us consider the function
\begin{equation*}
 \bar\zeta_{n}(x)=\bar w_n(x)-\underset{\p B_\eta}\min\,\bar w_n-\bar s_n(x),
\end{equation*}
which satisfies,
\begin{equation*}
 \begin{cases}
  -\D \bar\zeta_n=\bar W_n e^{\bar\zeta_n}\quad \text{in}\quad B_\eta, \\
  \int_{B_\eta}\bar W_ne^{\bar\zeta_n}\leq C,\\
  \bar\zeta_n=0 \quad \text{on}\quad \p B_\eta,
 \end{cases}
\end{equation*}
where
\begin{equation*}
 \bar W_n(x)=\bar V_n(x)e^{\bar\omega_n(x)}
\end{equation*}
and
\begin{equation*}
 \bar\omega_n(x)=\bar s_n(x)+\underset{\p B_\eta}\min \,\bar w_n.
\end{equation*}
We notice that,
\begin{equation*}
 \nabla\bar\omega_n\to\nabla \bar s\,\,\text{and}\,\, \nabla K_n\to\nabla K, \,\,\text{uniformly on compact sets of $B_\eta$}.
\end{equation*}
Moreover, since $\bar W_n(x)e^{\dis \bar\zeta_n}=\bar V_ne^{\dis \bar w_n}$, by Theorem \ref{Concentration}, we have that
\begin{equation}
 \label{vanishingappendix}
 \bar W_ne^{\dis \bar\zeta_n}\to0\qquad \text{uniformly on compact sets of $B_\eta\backslash\{0\}$}
\end{equation}
and
\begin{equation}
 \label{concentrationwnappendix}
 \bar W_ne^{\dis \bar\zeta_n}\rightharpoonup 8\pi\delta_{p=0}\qquad \text{weakly in the sense of measure in $B_\eta$}.
\end{equation}
Moreover,
\begin{equation}
 \label{concentration2appendix}
 (\frac{\epsilon_n^2}{|x_n|^2}+|z|^2)^{\bns}e^{\dis \bar\omega_n}e^{\dis \bar\zeta_n}\rightharpoonup\tfrac{8\pi}{K(0)}\delta_{p=0}\qquad \text{weakly in the sense of measure in $B_\eta$}.
\end{equation}
Let us set $\bar f_n(x)=\bar W_n(x)e^{\dis \bar\zeta_n(x)}$, then by Green's representation formula,
\begin{equation*}
 \bar\zeta_n(x)=-\tfrac{1}{2\pi}\int_{B_\eta}\ln (|x-y| )\bar f_n(y)\,dy+\int_{B_\eta}R(x,y)\bar f_n(y)\,dy,
\end{equation*}
where $R(x,y)$ is the regular part of Green's function associated to the Laplacian operator with Dirichlet boundary conditions on $B_\eta$.
Then, passing to the limit in the previous identity, we deduce that
\begin{equation}
 \label{wnconvergesappendix}
 \bar\zeta_n(x)\to -4\ln|x|+ g(x)\quad \text{in} \,\,C^1_{\rm loc}(B_\eta\backslash\{0\}),
\end{equation}
for some $g\in C^1(B_\eta)$. Let us set,
\[
 \bar\zeta_0(x)= -4\ln|x|+ g(x),
\]
we use the well known Pohozaev identity, for $s\in(0,\eta)$,
\begin{align}
 \nonumber
 \int_{\p B_s(0)}\Big((x,\nu)\tfrac{\dis |\grad \bar\zeta_n|^2}{2}-&(\nu,\grad \bar\zeta_n)(x,\grad \bar\zeta_n)\Big)d\sigma= \\ \label{Pohozaewnappendix}
 &=\int_{\p B_s(0)}(x,\nu)\bar W_n e^{\dis \bar\zeta_n}d\sigma-\int_{B_s}(2\bar W_n+x\cdot\grad \bar W_n)e^{\dis \bar\zeta_n}\,dx.
\end{align}

Firstly, we observe that
\begin{align*}
 &-\int_{B_s}x\cdot\grad \bar W_ne^{\dis \bar\zeta_n}\,dx= \\
 &=-\int_{B_s}x\cdot\grad (K_n(|x_n|x)e^{ \bar\omega_n})(\tfrac{\epsilon_n^2}{|x_n|^2}+|x|^2)^{\bns}e^{\dis \bar\zeta_n}-
 2\bns\int_{B_s}\tfrac{|x|^2}
 {\big(\frac{\epsilon_n^2}{|x_n|^2}+|x|^2\big)^{1-\bns}}K_n(|x_n|x)e^{ \bar\omega_n}e^{\dis \bar\zeta_n}\\
 &=-\int_{B_s}\tfrac{x\cdot\grad (K_n(|x_n|x)e^{ \bar\omega_n})}{K_n(|x_n|x)e^{ \bar\omega_n}}\bar W_ne^{\dis \bar\zeta_n}\,dx-2\bns\int_{B_s}\bar W_ne^{\dis \bar\zeta_n}\,dx+
 2\bns\int_{B_s}\tfrac{\frac{\epsilon_n^2}{|x_n|^2}}{\frac{\epsilon_n^2}{|x_n|^2}+|x|^2}\bar W_ne^{\dis \bar\zeta_n}\,dx.
\end{align*}

As in \cite{bt}, by using \eqref{vanishingappendix}, \eqref{concentration2appendix} and \eqref{wnconvergesappendix},  taking the limits as $n\to +\ii$ and $r\to 0$, we have that
$$
\int_{\p B_s(0)}\Big((x,\nu)\tfrac{\dis |\grad \bar\zeta_n|^2}{2}-(\nu,\grad \bar\zeta_n)(x,\grad \bar\zeta_n)\Big)d\sigma=-\tfrac{(8\pi)^2}{4\pi}+o(1)=-16\pi+o(1),
$$
and
$$
\int_{B_s}\tfrac{x\cdot\grad (K_n(|x_n|x)e^{ \bar\omega_n})}{K_n(|x_n|x)e^{ \bar\omega_n}}\bar W_ne^{\dis \bar\zeta_n}\,dx=o(1), \qquad \int_{\p B_s(0)}(x,\nu)\bar W_n e^{\dis \bar\zeta_n}d\sigma=o(1),
$$
$$
2\int_{B_s}\bar W_n e^{\dis \bar\zeta_n}\,dx=2 (8\pi) +o(1)=16\pi+o(1),
$$
$$
2\bns\int_{B_s}\bar W_n e^{\dis \bar\zeta_n}\,dx=2\bis(8\pi) +o(1)=16\pi+o(1).
$$
Now, let
\[
Q_{n,s}:=2\bns\int_{B_s}\tfrac{\frac{\epsilon_n^2}{|x_n|^2}}{\frac{\epsilon_n^2}{|x_n|^2}+|x|^2}\bar W_ne^{\dis \bar\zeta_n}\,dx.
\]
From the Pohozaev identity \eqref{Pohozaewnappendix} and the previous estimates, we deduce that,
\[
\lim_{s\to0^+}\lim_{n\to+\infty} Q_{n,s}=16\pi.
\]
On the other hand, we notice that
\[
Q_{n,s}=2\bns\int_{B_s}\tfrac{\frac{\epsilon_n^2}{|x_n|^2}}{\frac{\epsilon_n^2}{|x_n|^2}+|x|^2}\bar V_ne^{\dis \bar w_n}\,dx=2\bns\int_{B_{s|x_n|}(0)}\tfrac{\epsilon_n^2}{\epsilon_n^2+|x|^2}H_ne^{\dis v_n}\,dx\leq J_{n,r}.
\]
Therefore, taking first the limit for $n\to+\infty$ and then the limit as $s,r\to0^+$, we have that
\[
16\pi=\lim_{s\to0^+}\lim_{n\to+\infty}
Q_{n,s}\leq \lim_{r\to0^+}\lim_{n\to+\infty} J_{n,r}=J=0,
\]
which is a contradiction.

\end{proof}

\subsection{The proof of Lemma \ref{LS-quant}}\label{quant:III}$\,$\\

In this section we prove Lemma \ref{LS-quant}.
\begin{proof}
We will use $C>0$ to denote several constants whose value may change possibly line to line all along the proof.\\
We argue as in \cite{ls}. In view of \eqref{Kzero}, there exists $r_0>0$ such that $\min\limits_{\overline{B_{r_0}}} K>0$.
Let us fix any $r\in [\tfrac{r_0}{4},\tfrac{r_0}{2}]$ and set $s\in[8 d_n,r]$,
$$
\hat\O:=B_2\backslash B_{\frac{1}{2}},
$$
and
\[
 \hat v_n(x)=v_n(sx)+2(1+\bns)\log s,\quad x\in B_2.
\]
Then,
\[
 -\D \hat v_n(x)=\left({\frac{\epsilon_n^2}{s^2}+|x|^2}\right)^{\bns}K_n(sx)e^{\dis \hat v_{n}(x)}=:f_n(x) \quad\text{in}\quad B_2,
\]
and by \eqref{hypols-lem} and the fact that
$$
\tfrac{\epsilon_n}{s}\leq\tfrac{\epsilon_n}{8d_n}\leq\tfrac{C}{8},
$$
we have $\|f_n\|_{L^{\infty}(\hat \O)}\leq C$. Let $\kappa_n$ be the solution of
\begin{align*}
\begin{cases}
-\D \kappa_n=f_n & \hat\O, \\
\kappa_n=0 &\p\hat\O,
\end{cases}
\end{align*}
by standard elliptic estimates we have $\|\kappa_n\|_{L^\infty(\hat \O)}\leq C$. Then, once again by \eqref{hypols-lem},  the harmonic function $g_n=\kappa_n-\hat v_n$ is bounded from below by $-C$ and, by the Harnack inequality, there exists a universal constant $\tau_0\in(0,1)$ such that
\[
 \tau_0\sup_{\p B_1} (g_n+C)\leq \inf_{\p B_1}(g_n+C).
\]
Moving back to $\hat v_n$ we see that,
\[
 \sup_{\p B_1} \hat v_n\leq\tau_0 \inf_{\p B_1}\hat v_n+\hat C,
\]
which implies that,
\begin{equation}
 \label{HarnackI-lem}
 \sup_{\p B_s} v_n\leq\tau_0 \inf_{\p B_s}v_n-2(1-\tau_0)(1+\bns)\log  s+\hat C,
\end{equation}
for every $s\in[8d_n,r]$.
On the other hand, $$
\frac{\eps_n}{\delta_n}\leq C,\qquad \frac{|x_n|}{\dt_n}\leq C
$$
imply that we are in CASE (III). In particular $\bis=1$ and $\hat v_n$ satisfies the hypothesis of Lemma \ref{lemmasup+Cinfspecialcase}
with $\hat\epsilon_n=\tfrac{\epsilon_n}{s}$, $\hat\delta_n=\tfrac{\delta_n}{s}$ and $\hat x_n=\tfrac{x_n}{s}$.
Moreover, by this choice, we observe that $\tfrac{\hat\epsilon_n}{\hat\delta_n}=\tfrac{\epsilon_n}{\delta_n}$ and
$\tfrac{\hat x_n}{\hat\delta_n}=\tfrac{x_n}{\delta_n}$ and $\hat v_n(\hat\delta_n x)+2(1+\bns)\log\hat\delta_n=\widetilde v_n(x)$,
therefore $\widetilde \beta$ in \eqref{profile:tildew} is just the same as the
$\widetilde \beta$ appearing in Lemma \ref{lemmasup+Cinfspecialcase}. Hence, there exist two positive constants $C_0$ and $C_1$, (we notice $C_0$ depending only by $\widetilde\beta$), such that
\[
 \sup_{B_{\frac{1}{2}}} \hat v_n+C_0\inf_{B_1}\hat v_n\leq C_1,
\]
which immediately implies that,
\[
 \sup_{B_{\frac{s}{2}}} v_n+C_0\inf_{B_s} v_n\leq C_1-2(1+C_0)(1+\bns)\log s,
\]
for every $s\in[8d_n,r]$.  By noticing that  $ \sup\limits_{B_{\frac{s}{2}}} v_n= v_n(x_n)$, we can rewrite the above inequality as follows,
\begin{equation}
 \label{sup+infI-lem}
 \inf_{B_s} v_n\leq \tfrac{C_1}{C_0}-2(1+\tfrac{1}{C_0})(1+\bns)\log s-\tfrac{1}{C_0}v_n(x_n).
\end{equation}
Combining \eqref{HarnackI-lem} and \eqref{sup+infI-lem} and using the superharmonicity of $v_n$, we have that
\[
 \sup_{\p B_s}v_n\leq \tau_0\tfrac{C_1}{C_0}+\hat C-\tfrac{\tau_0}{C_0}v_n(x_n)-2(1+\tfrac{\tau_0}{C_0})(1+\bns)\log s,
\]
for every $s\in[8d_n,r]$, or equivalently,
\[
 e^{v_n(x)}\leq C_2\delta_n^{2(1+\bns)\tfrac{\tau_0}{C_0}}|x|^{-2(1+\bns)(1+\tfrac{\tau_0}{C_0})},
\]
for every $|x|\in[8d_n,r]$. In particular, this implies that
\begin{align*}
 &\int\limits_{B_{\frac{r_0}{2}}\setminus B_{8d_n}}H_n e^{\dis v_{n}}=\hspace{-.3cm}\int\limits_{B_{\frac{r_0}{2}}\setminus B_{8d_n}}(\eps_n^2+|x|^2)^{\bns} K_n e^{\dis v_{n}}\leq\\
 & C\hspace{-.3cm}\int\limits_{B_{\frac{r_0}{2}}\setminus B_{8d_n}}(\eps_n^2+|x|^2)^{\bns} \left(\frac{\delta_n}{|x|}\right)^{2(1+\bns)\tfrac{\tau_0}{C_0}} \left(\frac{1}{|x|}\right)^{2(1+\bns)}\leq\\
 & C\left(\frac{\delta_n}{d_n}\right)^{2(1+\bns)\tfrac{\tau_0}{C_0}}\underset{8\leq|z|\leq \tfrac{1}{d_n}}\int \left(\frac{\eps_n^2}{d_n^2}+|z|^2\right)^{\bns} \left(\frac{1}{|z|}\right)^{2(1+\bns)(1+\tfrac{\tau_0}{C_0})}\leq
 \\
 &C\left(\frac{\delta_n}{d_n}\right)^{2(1+\bns)\tfrac{\tau_0}{C_0}}\underset{8\leq|z|\leq \tfrac{1}{d_n}}\int \left(\frac{1}{|z|}\right)^{2}\left(\frac{1}{|z|}\right)^{2(1+\bns)\tfrac{\tau_0}{C_0}}\leq C\left(\frac{\delta_n}{d_n}\right)^{2(1+\bns)\tfrac{\tau_0}{C_0}}\to 0,
\end{align*}
where we used $\frac{\delta_n}{d_n}\to 0$.

At this point, by using also \eqref{-infinity}, we deduce that,
\begin{align*}
\lambda_n=\int\limits_{B_{8d_n}}H_n e^{\dis v_{n}}+\int\limits_{B_{\frac{r_0}{2}}\setminus B_{8d_n}}H_n e^{\dis v_{n}}+o(1)=\beta+o(1),
\end{align*}
as claimed.
\end{proof}

\subsection{The Proof of Lemma \ref{lem:barwn}}\label{sec:barwnstar}

By well known arguments (see for example \cite{BCW0}), we have that, for some $|z_0|=2$,
\begin{equation*}
 \bar w_n(x)=\tfrac{1}{2\pi}\int_{B_{\frac{r}{2|x_n|}}}\log(\tfrac{|z-z_0|}{|x-z|})\bar V_n e^{\bar w_n}\,dz-2(1+\bns)\log\bar \theta_n+O(1),
\end{equation*}
for any $4\leq |x| \leq \frac{r}{4|x_n|}$. Thus, we also have, setting $\xi_n=\frac{z_0}{d_n}$,

\begin{align*}
 w_n^*(y)&=\bar w_n(d_ny)+2(1+\bns)\log(d_n)\\
&=\tfrac{1}{2\pi}
\int\limits_{|z|\leq \frac{r}{2|x_n|}}\log(\tfrac{|z-z_0|}{|d_ny-z|})\bar V_n e^{\bar w_n}\,dz-2(1+\bns)\log\left(\frac{\bar \theta_n}{d_n}\right)+O(1)\\
&=\tfrac{1}{2\pi}
\int\limits_{d_n|\xi|\leq \frac{r}{2|x_n|}}\log(\tfrac{|d_n\xi-d_n\xi_n|}{|d_ny-d_n\xi|}) V^*_n e^{w_n^*}\,d\xi-2(1+\bns)\log\left(\frac{\bar \theta_n}{d_n}\right)+O(1)\\
&=\tfrac{1}{2\pi}
\int\limits_{|\xi|\leq \frac{r}{2d_n|x_n|}}\log(\tfrac{|\xi-\xi_n|}{|y-\xi|})V^*_n e^{w_n^*}\,d\xi-2(1+\bns)\log\left(\frac{\bar \theta_n}{d_n}\right)+O(1),
\end{align*}
for any $\frac{4}{d_n}\leq |y|\leq \tfrac{r}{4|x_n|d_n}$. Next, remark that for every $|y|\geq {2}$ and $|\xi|\leq 1$, we have that
\[
\tfrac{1}{|y-\xi|}\leq\tfrac{2}{|y|}\leq 1.
\]

Then, for every $|y|\geq {2}$,
\begin{align*}
\tfrac{1}{2\pi}\int_{B_{1}}\log(\tfrac{|\xi-\xi_n|}{|y-\xi|})V^*_n e^{w_n^*}\,d\xi&
\leq \omega_n+\tfrac{1}{2\pi}\log(\tfrac{2}{|y|})\int_{B_{1}}V^*_n e^{w_n^*}\,d\xi\\
&= (4+\tfrac{7}{2}\tau)\log(\tfrac{1}{|y|})+\omega_n + O(1),
\end{align*}
where
$$
\omega_n =\tfrac{1}{2\pi}\int_{B_{1}}\log({|\xi-\xi_n|})V^*_n e^{w_n^*}\,d\xi.
$$

In particular, for every $|y|\geq {2}$ and $1\leq |\xi|<\frac{|y|}{4}$, we have that
\[
\tfrac{|\xi-\xi_n|}{|y-\xi|}\leq\tfrac{2|y|}{3|y|}\leq 1,
\]
whence
\begin{align*}
\tfrac{1}{2\pi}\int_{B_{\frac{|y|}{4}}}\log(\tfrac{|\xi-\xi_n|}{|y-\xi|})V^*_n e^{w_n^*}\,d\xi
&\leq (4+\tfrac{7}{2}\tau)\log(\tfrac{1}{|y|})+\omega_n+O(1)+
\underbrace{\tfrac{1}{2\pi}\int_{{1\leq|\xi|\leq\frac{|y|}{4}}}\log(\tfrac{|\xi-\xi_n|}{|y-\xi|})
V^*_n e^{w_n^*}\,d\xi}_{\leq0}\\
&\leq  (4+\tfrac{7}{2}\tau)\log(\tfrac{1}{|y|})+\omega_n+O(1).
\end{align*}

Now, for $|y|\geq {2}$, we have that
\begin{align*}
 &\int_{B_{\frac{|y|}{4}}(y)}\log(\tfrac{|\xi-\xi_n|}{|y-\xi|})V^*_n e^{w_n^*}\,d\xi=\\
&= \underbrace{\int\limits_{B_{\frac{|y|}{4}}(y)}\log(|\xi-\xi_n|)V^*_n e^{w_n^*}\,d\xi}_{(A)}+
\underbrace{\int\limits_{|\xi-y|\leq\frac{1}{|y|^{1+\alpha_\infty}}}\log(\tfrac{1}{|y-\xi|})V^*_n e^{w_n^*}\,d\xi}_{(B)}+\\
&\hspace{8cm}+\underbrace{\int\limits_{\frac{1}{|y|^{1+\alpha_\infty}}\leq|\xi-y|\leq\frac{|y|}{4}}\log(\tfrac{1}{|y-\xi|})V^*_n e^{w_n^*}\,d\xi}_{(C)},
\end{align*}
where $\alpha_\infty=\al_{\ii,\sg}=1$.
Then,
\begin{align*}
 (A)+(C)&\leq (2+\alpha_\infty)\log(|y|)\int_{|\xi-y|\leq\frac{|y|}{4}}V^*_n e^{w_n^*}\,d\xi+O(1) \\
 &\leq 3\log(|y|)\int_{|\xi|\geq\frac{3}{2}}V^*_n e^{w_n^*}\,d\xi+O(1)\leq 4\pi \tau \log(|y|)+O(1),
\end{align*}
where we have used the fact that,
\[
 \int_{|\xi|\geq\frac{3}{2}}V^*_n e^{w_n^*}\,d\xi\leq \int_{|\xi|\geq 1}V^*_n e^{w_n^*}\,d\xi= \pi \tau + o(1).
\]

On the other hand, since $w_n^*\leq C$, we have that, putting $\alpha_n=\al_{n,\sg}$,
\begin{align*}
 (B)&\leq C_1 \int\limits_{\{|\xi-y|\leq\frac{1}{|y|^{1+\alpha_\infty}}\}}\log(\tfrac{1}{|\xi-y|})(1+|\xi|^2)^{\alpha_n}\\
 &\leq C_2 \int\limits_{\{|\xi-y|\leq\frac{1}{|y|^{1+\alpha_\infty}}\}}\log(\tfrac{1}{|\xi-y|})|\xi|^{2\alpha_n}\\
 &\leq C_3 (\tfrac{1}{|y|^{1+\alpha_\infty}}+|y|)^{2\alpha_n} \int\limits_{\{|\xi-y|\leq\frac{1}{|y|^{1+\alpha_\infty}}\}}\log(\tfrac{1}{|\xi-y|})
 \end{align*}
 \begin{align*}
 &\leq C_4 (\tfrac{1}{|y|^{1+\alpha_\infty}}+|y|)^{2\alpha_n} \tfrac{\log|y|}{|y|^{2(1+\alpha_\infty)}}=O(1),
\end{align*}
for $|y|\geq 1$, which implies that
\begin{equation*}
 \int_{B_{\frac{|y|}{4}}(y)}\log(\tfrac{|\xi-\xi_n|}{|y-\xi|})V^*_n e^{w_n^*}\,d\xi\leq  4\pi \tau \log|y|+O(1).
\end{equation*}

At last, if $\xi\in (B_{\frac{|y|}{4}}\cup B_{\frac{|y|}{4}}(y))^c$, we have that
\begin{equation*}
 \log(\tfrac{|\xi-\xi_n|}{|y-\xi|})\leq C.
\end{equation*}
Then, for every $|y|\geq {2}$, we see that
\begin{align*}\nonumber
  w_n^*(y)&\leq \omega_n-2(1+\bns)\log\left(\frac{\bar \theta_n}{d_n}\right)+(4+\tfrac{7}{2}\tau-2\tau)\log\frac{1}{|y|}+ O(1)\leq \\
 & \omega_n-2(1+\bns)\log\left(\frac{\bar \theta_n}{d_n}\right)+(4+\tau)\log\frac{1}{|y|}+ O(1)
\end{align*}
and consequently,
\begin{align*}
 V^*_n e^{w_n^*}\leq C\left(\frac{d_n}{\bar \theta_n}\right)^{2(1+\bns)}\frac{e^{\omega_n}}{|y|^{4-2\bns+\tau}},
\end{align*}
for every $|y| \geq {2}$ which proves the estimate from above in the statement of Lemma \ref{lem:barwn}.\\
At this point, we conclude the proof with the following estimate from below for $w_n^*$.\\
{\em
{\bf Claim:}
 \begin{align*}
 w_n^*(y)&\geq -\tfrac{M_n^*}{2\pi}\log(|y|)-2(1+\bns)\log\left(\frac{\bar \theta_n}{d_n}\right)+\omega_n + O(1),\quad |y|\geq 2,
\end{align*}
where
\[
M_n^*=\int_{B_{\frac{r}{2|x_n|d_n}}} V^*_n e^{w_n^*}\,d\xi
\]
and there exists $C>0$ such that,
$$
e^{\omega_n}\left(\frac{d_n}{\bar \theta_n}\right)^{2(1+\bns)}\leq C.
$$}
\noi
\!\!\!{\em Proof of Claim}.\\
Let us write,
\begin{align*}
w_n^*(y)=-\tfrac{M_n^*}{2\pi}\log(|y|)-2(1+\bns)\log\left(\frac{\bar \theta_n}{d_n}\right)+O(1)+
\tfrac{1}{2\pi}\int_{B_{\frac{r}{2|x_n|d_n}}}\log\left(\tfrac{|y||\xi-\xi_n|}{|y-\xi|}\right) V^*_n e^{w_n^*}\,d\xi
\end{align*}
and set $\om_1=B_{\frac{|y|}{2}}(0)$, $\om_2=B_{\frac{|y|}{2}}(y)$, $\om_3=(\om_1\cup \om_2)^c$.
Clearly for $\xi\in \om_1$ we have
$$
\frac23\leq \frac{|y|}{|y-\xi|}\leq 2,
$$
while for $\xi\in\om_2$
$$
\frac{|y|}{|y-\xi|}\geq 2,
$$
and for $\xi\in\om_3$,
$$
\frac{|y|}{3|\xi|}\leq \frac{|y|}{|y-\xi|}\leq 2.
$$
Therefore, with obvious meaning of notations, for some $|y|\geq 2$, we have
\begin{align*}
I(y)&=\tfrac{1}{2\pi}\int_{|\xi|\leq \frac{r}{2|x_n|d_n}}\log\left(\tfrac{|y|}{|y-\xi|}\right) V^*_n e^{w_n^*}\,d\xi=I_1(y)+I_2(y)+I_3(y)\\
&=O(1)+\tfrac{1}{2\pi}\int_{\om_2}\log\left(\tfrac{|y|}{2|y-\xi|}\right) V^*_n e^{w_n^*}\,d\xi+I_3(y)\\
&\geq O(1)+\tfrac{1}{2\pi}\int_{\om_3}\log\left(\tfrac{|y|}{3|\xi|}\right) V^*_n e^{w_n^*}\,d\xi\\
&\geq O(1)+\tfrac{1}{2\pi}\int_{|\xi|\geq \frac{|y|}{2}}\log\left(\tfrac{|y|}{|\xi|}\right) V^*_n e^{w_n^*}\,d\xi,
\end{align*}
while obviously we have,
$$
\tfrac{1}{2\pi}\int_{B_{\frac{r}{2|x_n|d_n}}}\log(|\xi-\xi_n|)V^*_n e^{w_n^*}\,d\xi=
\omega_n+\tfrac{1}{2\pi}\int_{|\xi|\geq 1}\log(|\xi|)V^*_n e^{w_n^*}\,d\xi+O(1).
$$

Next we deduce that,
\begin{align*}
&\tfrac{1}{2\pi}\int_{|\xi|\geq \frac{|y|}{2}}\log\left(\tfrac{|y|}{|\xi|}\right) V^*_n e^{w_n^*}\,d\xi+
\tfrac{1}{2\pi}\int_{|\xi|\geq 1}\log(|\xi|)V^*_n e^{w_n^*}\,d\xi=\\
&\tfrac{1}{2\pi}\log(|y|)\int_{|\xi|\geq \frac{|y|}{2}}V^*_n e^{w_n^*}\,d\xi+\tfrac{1}{2\pi}\int_{1\leq |\xi|\leq \frac{|y|}{2}}\log(|\xi|)V^*_n e^{w_n^*}\,d\xi\geq 0.
\end{align*}

As a consequence we have
$$
w_n^*(y)\geq -\tfrac{M_n^*}{2\pi}\log(|y|)-2(1+\bns)\log\left(\frac{\bar \theta_n}{d_n}\right)+\omega_n + O(1), \quad |y|\geq 2,
$$
and consequently, since $$\int_{\frac{r}{4|x_n|d_n}\leq |\xi|\leq \frac{r}{2|x_n|d_n}}V^*_n e^{w_n^*}\,d\xi=\int_{\frac{r}{4}\leq |x|\leq \frac{r}{2}}H_ne^{v_n}=o(1),$$
we have
$$
\pi \tau +o(1)=\int_{2\leq |\xi|\leq \frac{r}{4|x_n|d_n}}V^*_n e^{w_n^*}\,d\xi+o(1)\geq
C e^{\omega_n}\left(\frac{d_n}{\bar \theta_n}\right)^{2(1+\bns)}
\int_{2\leq |\xi|\leq \frac{r}{4|x_n|d_n}}\frac{|\xi|^{2\al_n}}{|\xi|^{4+\tfrac92 \tau}}\,d\xi=
$$
$$
C e^{\omega_n}\left(\frac{d_n}{\bar \theta_n}\right)^{2(1+\bns)}
\int_{2}^{\frac{r}{4|x_n|d_n}}s^{-4-\tfrac92 \tau+2\al_n+1}\,ds\geq C e^{\omega_n}\left(\frac{d_n}{\bar \theta_n}\right)^{2(1+\bns)}.
$$
This fact concludes the proof of the Claim.\finedim

\smallskip

\section*{Acknowledgements}
\noindent
Daniele Bartolucci and Paolo Cosentino were partially supported by the MIUR Excellence Department Project MatMod@TOV awarded to the Department of Mathematics, University of Rome ``Tor Vergata'', CUP E83C23000330006, 
by PRIN project 2022 2022AKNSE4, ERC PE1\_11, ``{\em Variational and Analytical aspects of Geometric PDEs}'', by the E.P.G.P. Project sponsored by the University of Rome ``Tor Vergata'', E83C25000550005, and by INdAM-GNAMPA Project, CUP E53C25002010001. \\
Lina Wu was partially supported by the National Natural Science Foundation of China (12201030).

\section*{Data availability statement}
Data sharing not applicable to this article as no datasets were generated or analysed during the current study.

\section*{Conflict-of-Interest Statement}
The authors declare that they have no conflict of interest.

\section*{Statement on the use of AI}
After completing the manuscript, the authors used ChatGPT (GPT-6 and GPT-6 PRO) to assist in checking mathematical arguments and proofs. During this process, some of the proofs previously developed by the authors were simplified and therefore incorporated into the final version of the paper. In particular, AI assistance helped the authors rule out the possibility $\lm_\ii=16\pi$ in CASE (I) of Theorem \ref{quant:imp}, establish that the constant $C_{III}$ in CASE (III) of Theorem \ref{supinfI} can be chosen independently of $\tilde C$ and simplify certain arguments in the proof of the Harnack-type inequality (Theorem \ref{supinfI}).\\

All mathematical arguments and the final version of the manuscript have been independently reviewed and verified by the authors, who take full responsibility for the content of the paper.

\end{document}